\documentclass[11pt,pdflatex, makeidx]{amsart}
\usepackage{mathpazo}
\usepackage{bm,mathrsfs}
\usepackage{enumerate}
\usepackage[table]{xcolor}
\usepackage{setspace}

\usepackage{pdfpages} 
\usepackage[colorlinks=true, linkcolor=blue]{hyper ref}
\usepackage{multicol} 

\usepackage{ulem}
\usepackage{amssymb}
\usepackage{graphicx, epsfig}
\usepackage{latexsym, amsfonts, amscd, amsmath}
\usepackage{mathrsfs}
\usepackage{lscape}
\usepackage{wrapfig}
\usepackage{tabularx}
\makeatletter
\def\hlinewd#1{%
\noalign{\ifnum0=`}\fi\hrule \@height #1 %
\futurelet\reserved@a\@xhline}
\makeatother

\makeindex 
\input xy
\xyoption{all}
\usepackage{wallpaper}

\definecolor{orange}{rgb}{1,0.5,0}
\definecolor{Indigo}{rgb}{0.2,0.1,0.7}
\definecolor{Violet}{rgb}{0.5,0.1,0.7}

\newtheorem{thm}{Theorem}[subsection]
\newtheorem{prop}[thm]{Proposition}
\newtheorem{lem}[thm]{Lemma}
\newtheorem{cor}[thm]{Corollary}

\theoremstyle{definition}
\newtheorem{dfn}[thm]{Definition}

\theoremstyle{remark}
\newtheorem{rmk}[thm]{Remark}

\newenvironment{mat}{\left(\begin{smallmatrix}
\end{smallmatrix}\right)}{enddef}

\newcommand{\Aut}{{\operatorname{Aut}}}

\newcommand{\End}{{\operatorname{End}}}

\newcommand{\Hom}{{\operatorname{Hom}}}

\newcommand{\Ker}{{\operatorname{Ker}}}

\newcommand{\ord}{{\operatorname{ord }}}
\newcommand{\Pic}{{\operatorname{Pic }}}

\newcommand{\Spec}{{\operatorname{Spec }}}
\newcommand{\Stab}{{\operatorname{Stab }}}

\newcommand{\SL}{{\operatorname{SL }}}
\newcommand{\GL}{{\operatorname{GL}}}
\newcommand{\Gal}{{\operatorname{Gal}}}

\newcommand{\PGL}{{\operatorname{PGL}}}

\newcommand{\gerd}{{\frak{d}}}

\newcommand{\gerf}{{\frak{f}}}

\newcommand{\germ}{{\frak{m}}}

\newcommand{\gerp}{{\frak{p}}}

\newcommand{\uE}{{\underline{E}}}

\newcommand{\calC}{{\mathcal{C}}}

\newcommand{\calE}{{\mathcal{E}}}

\newcommand{\calO}{{\mathcal{O}}}
\newcommand{\calP}{{\mathcal{P}}}

\newcommand{\calR}{{\mathcal{R}}}

\def\AA{\mathbb{A}}

\def\CC{\mathbb{C}}

\def\FF{\mathbb{F}}
\def\GG{\mathbb{G}}
\def\HH{\mathbb{H}}

\def\KK{\mathbb{K}}

\def\NN{\mathbb{N}}

\def\PP{\mathbb{P}}
\def\QQ{\mathbb{Q}}

\def\ZZ{\mathbb{Z}}

\newcommand{\scrA}{{\mathscr{A}}}
\newcommand{\scrB}{{\mathscr{B}}}
\newcommand{\scrC}{{\mathscr{C}}}

\newcommand{\scrE}{{\mathscr{E}}}
\newcommand{\scrF}{{\mathscr{F}}}
\newcommand{\scrG}{{\mathscr{G}}}
\newcommand{\scrH}{{\mathscr{H}}}

\newcommand{\scrL}{{\mathscr{L}}}

\newcommand{\scrR}{{\mathscr{R}}}

\newcommand{\id}{{\noindent}}

\newcommand{\arr}{{\; \rightarrow \;}}
\newcommand{\Arr}{{\; \longrightarrow \;}}

\newcommand{\injects}{{\; \hookrightarrow \;}}

\newcommand{\fpbar}{\overline{\mathbb{F}}_p}
\newcommand{\Spf}{\operatorname{Spf }}

\newcommand{\spe}{{\operatorname{sp }}}

\newcommand{\Fbar}{\overline{F}}

\newcommand{\LambdaEN}{\Lambda_\ell(\uE,N)}

\newcommand{\LambdaE}{\Lambda_\ell(E,1)}
\newcommand{\LambdaENv}{\Lambda_\ell(\uE,N) ^V}

\newcommand{\LambdaEundir}{\Lambda_\ell^{{\rm ud}}(E,1)}

\newcommand{\Cbar}{\overline{C}}

\newcommand{\Walk}{{\rm Walk}}

\newcommand{\Abar}{\overline{A}}

\begin{document}
\marginparwidth 50pt

\title {$p$-adic Dynamics of Hecke Operators on Modular Curves}
\author{Eyal Z. Goren \& Payman L Kassaei}
\address{Department of Mathematics and Statistics, McGill University, 805 Sherbrooke St. W., Montreal H3A 0B9, QC,
Canada.}
\address{Department of Mathematics, King's College London, Strand, London WC2R 2LS, UK} \email{eyal.goren@mcgill.ca;
payman.kassaei@kcl.ac.uk} 
\subjclass{Primary 11F32, 11G18;  Secondary
11G15, 14G35}
\begin{abstract} In this paper we study the $p$-adic dynamics of  prime-to-$p$ Hecke operators on the set of points of modular curves in both cases of good ordinary and supersingular reduction. We pay special attention to the dynamics on the set of CM points. In the case of ordinary reduction we employ the Serre-Tate  coordinates, while in the case of supersingular reduction we use a parameter on the deformation space of the unique formal group of height $2$ over $\fpbar$, and take advantage of the Gross-Hopkins period map.    
\end{abstract}
\maketitle
\setcounter{tocdepth}{2}

\section{Introduction}
\id In this paper we study the dynamics of Hecke operators on the modular curves $X_1(N)$ in the $p$-adic topology. The motivation for this paper initially came from a desire to understand the impact of this dynamics on the study of the canonical subgroup over Shimura varieties, as part of the continuation of our work \cite{GK}. While studying this problem, we realized that even in the simplest case of Shimura varieties, the modular curves, there are significant challenges. Our work on this specific setting is the content of this paper; the techniques introduced here will be useful in studying much more general situations that we hope to discuss in future work. 

There are other motivations to study this problem.  Let $R$ be the maximal order of a finite field extension of $\QQ_p$. Let $\kappa$ denote  the residue field of $R$. Let $X$ be a Shimura variety of PEL type with a smooth integral model over $R$, over which the action of the prime-to-$p$ Hecke algebra extends.  The action on the special fibre $X_\kappa$ has been studied extensively by Chai and Oort, as part of Oort's philosophy of foliations on Shimura varieties of PEL type \cite{Chai, CO2, Oort}. From our perspective, the work of Chai and Oort is a ``mod $p$ shadow" of the $p$-adic dynamics of Hecke operators.

Secondly, Clozel and Ullmo \cite{ClozelUllmo1, ClozelUllmo2} have studied dynamics of Hecke operators over the complex numbers in connection with the Andr\'e-Oort conjecture, bringing in tools from measure theory and ergodic theory. Our work resonates well with theirs as our main interest lies in the action of Hecke operators on special subvarieties, albeit in the $p$-adic topology. In particular, in this paper we pay special attention to the action of Hecke operators on CM points. More than that, somewhat serendipitously, after  non-trivial reductions, including  via the Gross-Hopkins period morphism, our analysis uses results on random walks on groups and ergodic theory. 

When this work was essentially completed, we learned of interesting work on the same topic by  Herrero-Menares-Rivera-Letelier \cite{HMR1}; see also the forthcoming \cite{HMR2}. The points of view taken in our respective works are different, and the results fit well together to shed light on complementary aspects of the theory. While we are interested in the distribution of the set of points in a Hecke orbit in the $p$-adic topology, {\it loc. cit.}  considers the convergence of a sequence of Hecke orbits to the Gauss point of the affine Berkovich line. Unlike our work, \cite{HMR1} considers only modular curves of level $1$, but it is believed that their results hold in more general levels. See also the recent preprint \cite{Dis}. 


\

We now briefly discuss the contents of this paper. Our main interest is in the action of the prime-to-$p$ Hecke operators on CM points (points on the modular curve that correspond to CM elliptic curves) regarded as the special Shimura subvarieties of the modular curve $X_1(N)$. We denote a point of $X_1(N)(\CC_p)$ by $x = (E_x, P_x)$, where $E_x$ is an elliptic curve and $P_x$ a point of order $N$ of it. We simplify the analysis by looking at the action of the (iterations of) a single Hecke operator $T_\ell$, where $\ell \neq p$ is a prime. A key point of our approach is to use reduction mod $p$ to propel results from characteristic~$p$ to characteristic~$0$. 
The analysis breaks naturally into two cases: (i) ordinary reduction and (ii) supersingular reduction. That is, we distinguish between the case where the elliptic curve $E_x$ has good ordinary reduction, and the case where it has supersingular reduction. The case of bad reduction is not important to us as our main interest is in CM points; these always have (potentially) good reduction. That being said, the methods in the ordinary case can easily be applied to the case of bad reduction.

\medskip

\id {\bf Ordinary reduction.}  The locus of ordinary reduction in the modular curve $X_1(N)$ can be partitioned into a union of residue discs which are shuffled around via the Hecke operator $T_\ell$ in a way that can be fully understood using the theory of $\ell$-isogeny volcanoes. In fact, the centre of these discs, i.e., the {\it canonical lifts}, are moved around exactly according to the $\ell$-isogeny volcano. Consequently, to study the dynamics of $T_\ell$ acting on a point $x$ of ordinary reduction, one needs to  understand the returns of the orbit to the residue disc $D$ containing $x$. It transpires that this question can be reduced to the dynamics of a {\it single} automorphism of the disc, induced by a particular endomorphism $f$  of $\bar E$, the reduction of $E_x$. At this point, we switch gears and, by means of Serre-Tate theory, view the residue disc as the rigid-analytic generic fiber of the formal scheme that is the universal deformation space of $\bar E$. As such, the disc has a natural parameter $t$, where the canonical lift corresponds to $t = 0$. We prove (Theorem~\ref{thm: dynamics on the disc}), using \cite{Katz}, that the action of $f$ on the disc is given by 
\[ t \mapsto (1+t)^{\bar f/f} - 1,\]
where $f$ and its complex conjugate $\bar f$ are viewed as elements of $\ZZ_p^\times$ via their action on the $p$-adic Tate module of $\bar E$. This provides a very clear picture of the orbits and their closures and allows us to completely determine the dynamics of $T_\ell$ on the CM points in $D$, which turn out to be periodic points for the action of $f$; see \S \ref{sec: dynamics on ordinary}, in particular, Propositions~\ref{prop: 1234}, \ref{cor: summary of T ell orbit}, \ref{prop: 1235},  \ref{prop: 1236}. 

\medskip

Throughout our discussion  in \S\S \ref{section:mod p isogeny graph} - \ref{sec: dynamics on ordinary} of the action of Hecke operators on elliptic curves $E$ with ordinary reduction $\bar E$, we assume that 

\id  \begin{center} $(\star) \qquad \End^0(E) \not\cong \QQ(i), \QQ(\rho),$\end{center} 

\id where $i^2 = -1, \rho^2 + \rho+1 = 0$. This guarantees that for any elliptic curve $\bar E^\prime$ isogenous to $\bar E$, we have $\Aut(\bar E^\prime) = \{ \pm 1\}$. This assumption is made for simplicity only; the methods extend to the general case.

 \

\id {\bf Supersingular reduction.} Once more, our analysis begins in characteristic $p$, where we introduce a class of directed graphs $\Lambda_\ell(N)$ that we call ``supersingular graphs". These graphs are defined in a way that encodes the action of $T_\ell$ on the supersingular points in $X_1(N)(\fpbar)$. Roughly speaking, vertices of $\Lambda_\ell(N)$ correspond to supersingular points of $X_1(N)(\fpbar)$, and edges to  isogenies of degree $\ell$. They are directed graphs of outgoing degree $\ell+1$, and usually of ingoing degree $\ell+1$ as well. They serve as level $\Gamma_1(N)$ analogues of the (graphs defined by the) Brandt matrix $B(\ell)$. 

Again, the finitely many residue discs of supersingular reduction in $X_1(N)$ move under the iterations of $T_\ell$ according to walks in the graph $\Lambda_\ell(N)$. However, the returns under $T_\ell$ to a particular supersingular disc  turns out to be much more complicated than in the ordinary case. Fixing a supersingular residue disc $D\subset X_1(N)(\CC_p)$ specializing to  a supersingular elliptic curve $(\bar{E},\bar{P})$ in characteristic $p$, the branches of the iterations of $T_\ell$ that return to $D$ generate a subgroup $\scrH_N \subset (\End(\bar{E})\otimes \ZZ_p)^\times=:\scrR^\times$, $\scrR$ being the maximal order in the unique non-split quaternion algebra over $\QQ_p$. 
 The nature of the groups $\scrH_N$ is not entirely clear. We prove several results about these groups and, in particular, show that they have an abundance of mutually non-commuting elements; enough so that they are dense in $\GL_2$ relative to the Zariski topology, under the inclusion $\scrR^\times \injects \GL_2(\QQ_q)$, where $q = p^2$, and $\QQ_q = W(\FF_q)[p^{-1}]$.

To study the action of the highly complicated non-abelian  group $\scrH_N$ on $D$, we once again interpret  $D$ as the rigid-analytic space associated to a universal deformation space, this time of a $1$-dimensional formal group of height $2$. This allows us to use the work of Gross-Hopkins \cite{GH} and Lubin-Tate \cite{LT}. Using the equivariant Gross-Hopkins period morphism $\Phi: D \arr \PP^1$, we may study of the action of $\scrH_N$ on $D$ via the action of  M\"obius transformations on $\PP^1$. We define open discs $J \subset D, U \subset \PP^1$ such that $\Phi$ restricts to an isomorphism $J \arr U$. We show that all CM points of discriminant prime to $p$ lie in $J(W(\FF_q))$, which is in bijection with $U(W(\FF_q))$, the open unit disc in $W(\FF_q)$.  The key point now is that we are able to obtain a very explicit description of the $p$-adic closure $\HH_N$ of $\scrH_N$ in $\scrR^\times$.  We equip $\HH_N$ with a measure coming from the action of~$T_\ell$.  We translate our questions about the Hecke orbit of a CM point $x\in D$ under $T_\ell$ to questions about stationary measures for the action of $\HH_N$. Using the description of $\HH_N$ and the work of Benoist-Quint \cite{BQ1}, we conclude that there is a unique stationary probability measure (namely, a measure fixed under the action of $\HH_N$), independent of $x$ and~$N$, whose support is $J(W(\FF_q))$. See Theorems~\ref{cor: 5104} - \ref{thm:5107}.

\section{The $\ell$-isogeny graph: the ordinary case}\label{section:mod p isogeny graph}
Before defining the various $\ell$-isogeny graphs of elliptic curves, we recall some definitions from graph theory.

\subsection{Graph theory terminology} Let $\Lambda$ be a directed graph. We will call a directed edge, $v \arr w$, an {\it arrow}. A {\it walk} in $\Lambda$ is a sequence of arrows $v_{i} \overset{e_i}{\arr} w_i$ for $i=1,..,d$, where $w_i=v_{i+1}$ for $i=1,...,d-1$.   This walk is called a closed walk if $v_1=w_d$.  We allow the empty walk and consider it a closed walk. If $\omega_1$, $\omega_2$ are walks in $\Lambda$ such that the end point of $\omega_1$ is the starting point of $\omega_2$, we denote $\omega_2 \circ \omega_1$ the walk obtained by traversing $\omega_1$ followed by $\omega_2$. If $v$ is a vertex in $\Lambda$, we define $\Omega(\Lambda,v)$ to be the monoid of closed walks in $\Lambda$ starting and ending at $v$ under $\circ$. We often write $\omega_2\omega_1$ for $\omega_2 \circ \omega_1$.

If $\omega$ is a closed walk in $\Lambda$ starting at $v_1$ and traversing, in order, vertices $v_1,v_2,\cdots,v_d=v_1$, we can view $\omega$ as a closed walk in $\Omega(\Lambda,v_i)$ for all $1 \leq i \leq d$. We will often abuse notation and denote all these walks by the same notation if the starting point of the walk is understood.

A \textit{path} is a walk in which vertices do not repeat (and so a path is necessarily open). In a \textit{cycle}, vertices do not repeat, except that the last vertex is equal to the first (and so a cycle is closed). 

An undirected graph is called {\it regular} if every vertex has the same degree. Every loop, i.e., an edge between a vertex and itself, contributes one to the degree at the vertex. A directed graph is called \textit{regular of degree $r$} if the in-degree and the out-degree of each vertex equal $r$.  An arrow of the form $v \arr v$ contributes one to both the in-degree and the out-degree. A directed graph is called \textit{connected} if given any two vertices $x, y$, there is a walk starting from $x$ and ending in $y$.

Let $\ell$ be a positive integer. An infinite connected undirected graph $\Sigma$ is called an $\ell$-{\it volcano} if there is a surjective function $b\colon \Sigma^V \arr \NN$ (where $\Sigma^V$ is the set of vertices) such that
\begin{enumerate}
\item $\Sigma$ is $(\ell +1)$-regular,
\item $b^{-1}(0)$ with its induced subgraph structure (called the {\it rim}) is a regular graph of degree at most $2$,
\item for $i>0$, each vertex in $b^{-1}(i)$ is connected to a unique vertex in $b^{-1}(i-1)$ and to~$\ell$ vertices in $b^{-1}(i+1)$.
\end{enumerate}

\subsection{Definition of the ordinary $\ell$-isogeny graphs}
Let $p,\ell$ be distinct primes. Let $N\geq 1$ be an integer such that $(N,p\ell)=1$. In the following we will study the $\ell$-isogeny graph of $\uE=(E,P)$, consisting of an ordinary elliptic curve $E$ over $\fpbar$ with a point $P$ of order  $N$ on it (cf. \cite{Sutherland} in the case $N=1$). Recall our standing assumption: $(\star)\;\; \End^0(E) \not\cong \QQ(i), \QQ(\rho)$. 
For two such objects, $\uE_1=(E_1,P_1)$ and $\uE_2=(E_2,P_2)$, we write 
\[ \uE_1 \overset{\ell}{\sim} \uE_2 \]
if there is an isogeny $\lambda\colon E_1 \arr E_2$ of degree a power of $\ell$ such that $\lambda(P_1)= P_2$. This defines an equivalence relation on the set of isomorphism classes of all such pairs (to see the symmetry, let $\alpha(N)$ denote the order of $\ell$ in $(\ZZ/N\ZZ)^\times$. Then  $\deg(\lambda)^{\alpha(N)-1}\lambda^\vee\colon \uE_2 \arr \uE_1$ takes $P_2$ to $P_1$). 

We define a directed graph $\LambdaEN$ as follows. 
The set of vertices $\LambdaENv$ is the set of isomorphism classes of all pairs $\uE_1=(E_1,P_1)$ over $\fpbar$ that satisfy $\uE_1  \overset{\ell}{\sim} \uE$. We define an arrow between two such isomorphism classes $[(E_1,P_1)]$, $[(E_2,P_2)]$ to be an isogeny $\lambda\colon (E_1,P_1) \arr (E_2,P_2)$ of degree $\ell$, where we identify two such isogenies up to isomorphism. In other words,  arrows corresponding to $\lambda\colon (E_1,P_1) \arr (E_2,P_2)$ and $\lambda'\colon (E_1' ,P_1')\arr (E_2',P_2')$ are identified if there is a commutative diagram
\[ 
\xymatrix{
(E_1,P_1) \ar[d]_{f_1}\ar[r]^{\lambda} & (E_2,P_2) \ar[d]^{f_2} \\
(E_1',P_1') \ar[r]^{\lambda'} & (E_2',P_2')
}
\]
where $f_1$, $f_2$ are isomorphisms. Note that $\LambdaEN$ is a connected directed graph. Also, due to our assumption on $\Aut(E)$, when $N >2$, once we fix representatives $(E_1,P_1)$ and $(E_2,P_2)$ above, an arrow between the corresponding vertices determines a unique isogeny of degree $\ell$ from $(E_1,P_1)$ to $(E_2,P_2)$. When $N\leq 2$, this isogeny is only defined up to sign. We further define
\[
\Lambda_\ell^{\rm ord}(N)=\coprod_{\uE\ {\rm ordinary}} \LambdaEN,
\]
where $\uE$ runs over all representatives $\uE=(E,P)$ modulo the equivalence relation $\overset{\ell}{\sim}$.

Let $K=\End^0(E)$ be the algebra of rational endomorphisms of $E$, which is a quadratic imaginary field of discriminant $\Delta=\Delta(E)$ ($\Delta \neq -3, -4$).  We denote the order of conductor $c$ in $K$ by~$R_c$. If $\uE \overset{\ell}{\sim} \uE_1$ then $\End(E_1)$ is an order in $K$  of conductor 
\[
\gerf(E_1)=\ell^r m,
\]
 where $(m,p\ell)=1$, and $r \geq 0$. Note that $m$ is an invariant of the equivalence classes under $\overset{\ell}{\sim}$. The discriminant $\gerd(E_1)$ of $\End(E_1)$ satisfies $\gerd(E_1)=\gerf(E_1)^2\Delta(E_1)$. We can define a level function  
 \[
 b_N\colon \LambdaENv \arr \NN
 \] 
by setting $b_N([\uE_1])=r$.

\subsection{The isogeny associated to a walk}   Let $\omega$ be a walk $[\uE_0] \overset{\lambda_1}{\arr} [\uE_1]  \overset{\lambda_2}{\arr} ...  \overset{\lambda_d}{\arr} [\uE_d]$ in $\LambdaEN$.
Note that if $N>2$ we can think of each $\lambda_i$ as an isogeny of degree $\ell$ from $\uE_{i-1}$ to $\uE_i$; if $N\leq 2$, these isogenies are only defined up to sign. We define 
\[
\tilde{\omega}:=\lambda_d\circ...\circ\lambda_2\circ\lambda_1
\]
which is an $\ell^d$-isogeny from $E_0$ to $E_d$ sending $P_0$ to $P_d$, defined only up to sign if $N \leq 2$. In particular, if $\omega \in \Omega(\Lambda_\ell(\uE,N),[\uE])$ then $\tilde{\omega}$ is an element of $\End(E)$ of degree $\ell^d$ fixing $P$.

\subsection{Properties of the $\ell$-isogeny graphs} When $N=1$, the directed graph $\LambdaE$ has the property that  for every arrow $[E_1] \arr [E_2]$ corresponding to an $\ell$-isogeny $\lambda\colon E_1\arr E_2$, there is a corresponding arrow $[E_2] \arr  [E_1]$ given by the dual isogeny $\lambda^\vee\colon  E_2 \arr E_1$. Identifying these corresponding arrows, we obtain an {\textit{undirected}} graph denoted $\LambdaEundir$. Our convention is that if $[E_1]=[E_2]$, and if $\lambda^\vee \neq \pm\lambda$, then the resulting loop is counted with multiplicity two in $\LambdaEundir$; otherwise, it is counted with multiplicity~$1$.  The function $b_1$ defined above descends to $b_1\colon\LambdaEundir \arr \NN$. 

For a set $S$ of endomorphisms of $E$ (and in particular for a prime ideal $L$ of $\End(E)$) we let $E[S] = \cap_{s\in S} \Ker(s)$. The following proposition about the structure of the graph $\LambdaEundir$ is known (cf. \cite{Sutherland}).
 
 \begin{prop}\label{prop:volcano} Let $E$ be an ordinary elliptic curve over $\fpbar$ with $\End^0(E)\not\cong \QQ(i), \QQ(\rho)$ and $\mathfrak{f}(E) = \ell^rm$. 
 \begin{enumerate}
 \item The graph $\LambdaEundir$ is an $\ell$-volcano with respect to the level function $b_1$. The rim is a regular graph of degree $1+(\frac{\Delta(E)}{\ell})$, the vertices of which are exactly those vertices of $\LambdaEundir$ that have CM by $R_m$. (Here $(\frac{\Delta(E)}{\ell})$ is the Kronecker symbol.)
 
\item  Let $L|\ell$ be a prime ideal of $R_m$. Let $a=ord_{cl(R_m)}(L)$; hence we have  $L^a=(f)$ for $f \in R_m$. If $\ell$ is inert in $R_m$, then the rim consists of a single vertex and has no edges. If $\ell$ is ramified in $R_m$ but $L$ is not principal, then the rim consists of two vertices $[E_1]$ and $[E_1/E_1[L]]$ with exactly one edge between them - the class of the  natural projection $E_1 \arr E_1/E_1[L]$. In all other cases,  the rim consists of a cycle of length $a$ given by $$[E_1] - [E_1/E_1[L]] - \cdots - [E_1/E_1[L^a]] = [E_1]$$ (using the isomorphism $E_1/E_1[L^a] {\arr} E$ induced by $f$).
\item  The cardinality of $b_1^{-1}(i)$ equals 
\[
|b_1^{-1}(i)|=\begin{cases} ord_{cl(R_m)}(L) & i=0,\ L|\ell\ {\rm a \ prime\ ideal\ of}\ R_m,\\  (\ell-(\frac{\Delta(E)}{\ell}))\ell^{i-1}\cdot ord_{cl(R_m)}(L)  & i\geq 1.  \end{cases}
\]
 \end{enumerate}
 \end{prop}
 
 If $[E'] \overset{\lambda}{\arr} [E'']$ is an arrow in $\Lambda_\ell(E,1)$, we denote $[E'']  \overset{\lambda^\vee}{\arr} [E']$  the arrow obtained via the dual isogeny.  Similarly, we define for a walk $\omega$ from $[E']$ to $[E'']$, a walk $\omega^\vee$ in the reverse direction.  
 
 We define two particular  walks in $\LambdaE$ as follows.
  If $\ell$ is inert in $R_m$, we define $\omega_{rim}^+=\omega_{rim}^-$ to be the empty walk. If $\ell = L^2$ is ramified in $R_m$ we define $\omega_{rim}^+=\omega_{rim}^-$ to be the unique closed walk in the rim (up to  the choice of starting point; see Remark \ref{remark:cycle} below). It is a loop if $L$ is principal and of length $2$ if it is not.   If $\ell$ is split, there are two distinct cycles in $\Lambda_\ell(E,1)$ traversing the rim of $\LambdaEundir$, one in each possible direction. Let us call them $\omega_{rim}^+$ and $\omega_{rim}^-$.  We denote  
$f_{rim}^+,f_{rim}^- \in R_m$ the isogenies associated to $\omega_{rim}^+$ and $\omega_{rim}^-$, respectively (with the usual sign ambiguity).
 
\begin{rmk}\label{remark:cycle} As explained at the beginning of this section, we are using the same notation, $\omega_{rim}^{\pm}$, to denote the several walks that start and end at any vertex along the rim. This shall not cause confusion as it is always understood what starting point is considered. Regardless of the starting point of $\omega_{rim}^+$, $f_{rim}^+$
will be the same element of $R_m$ under the natural identifications of the endomorphism rings of the elliptic curves on the rim.
 \end{rmk}

Consider the homotopy equivalence relation on the set of walks in $\Lambda_\ell(E,1)$ generated by relations
\[
\omega_2\omega_1 \sim_h \omega_2\omega^\vee \omega \omega_1,
\] 
where $\omega_1,\omega_2$ are two walks such that the endpoint of $\omega_1$ is the starting point of $\omega_2$ (call it $[E_1]$), and  $\omega$ is a walk starting at $[E_1]$. It is easy to see that Proposition \ref{prop:volcano} implies the following.

 \begin{cor} \label{cor:walks} Every element in $\Omega(\Lambda_\ell(E,1),[E_1])$ is equivalent under $\sim_h$ to a closed walk of the form
 \[
 \omega^\vee(\omega_{rim}^{\pm})^n\omega,
 \]
 where $\omega$ is a path from $[E_1]$ to the rim of the graph, and $n$ is a non-negative integer. 
\end{cor}

Consider the map of directed graphs $$\pi_N\colon \LambdaEN \arr \LambdaE$$ which sends a vertex $[(E_1,P_1)]$  to $[E_1]$, and an arrow $[(E_1,P_1)] \overset{\lambda}{\arr} [(E_2,P_2)]$ to $[E_1] \overset{\lambda}{\arr} [E_2]$. 
Note that for any $N$, and any vertex $[\uE_1]\in \LambdaENv$, the arrows starting in $[\uE_1]$ are in natural bijection with subgroups of order $\ell$ of $E_1$ and so $\pi_N$ is a \textit{covering map} of directed graphs. As such, many properties of $\LambdaEN$ can be deduced from the corresponding properties of $\LambdaE$. For example, since the action of the operator $T_\ell$ on a point $[\uE_1]$ can be described by $[\uE_1] \mapsto \sum_{[\uE_1] \arr [\uE_1^\prime]} [\uE_1^\prime]$ (sum over arrows in $\Lambda_\ell^{\rm ord}(N)$), the action of $T_\ell$ on $[\uE_1] = [(E_1, P_1)]$ and $[E_1]$ are easily related.

\section{The Hecke correspondence $T_\ell$ on ordinary elliptic curves}

\subsection{$T_\ell$ in characteristic $0$} For what follows, the reader may consult \cite{CO1} for a general introduction to deformations of abelian varieties and $p$-divisible groups, and \cite{Katz} for the case of ordinary abelian varieties. In addition, the reference \cite{Con} provides a good introduction to rigid-analytic geometry, while the construction of the generic fibre of a formal scheme appears in \cite{Ber}.

\medskip 

\id Recall that $p,\ell$ are distinct primes, and $N$ is an integer coprime to $p\ell$. Let $E$ be an ordinary  elliptic curve defined over $\fpbar$. Assume that $\Aut(E)=\{\pm1\}$. Let $\tilde{D}(E)\cong \Spf(W(\fpbar)[[x]])$ be the deformation space of $E$ defined over $W(\fpbar)$; it pro-represents the functor of deformations of $E$ to local artinian rings $(R,\germ_R)$ with a given isomorphism $R/\germ_R\cong \fpbar$. Let $D(E)=\tilde{D}(E)_{rig}$ denote the associated rigid-analytic space; it is isomorphic to the open unit disc over $W(\fpbar)[1/p]$. For an elliptic curve $\tilde{E}$ lifting $E$, we define $D(\tilde{E})=D(E)$.
 
The same construction can be done for a pair $\uE=(E,P)$, where $E$ is as above, and $P$ is a point of order  $N$ on $E$. Let $\tilde{D}(\uE)$ denote the formal deformation space of $\uE$ over $W(\fpbar)$, and let $D(\uE)$ be the associated rigid-analytic space.  For $\tilde{\uE}$ lifting $\uE$, we set $D(\tilde{\uE})=D(\uE)$.

Since $E[N]$ is an \'etale group scheme over $\fpbar$, the forgetful map provides an isomorphism between the deformation space of $\uE$ and $E$. Hence, the rigid-analytic forgetful map
\[
\pi_N\colon D(\uE) \arr D(E)
\]
is an isomorphism too.

When $N \geq 4$, the moduli of pairs $(\tilde{E},\tilde{P})$ where $\tilde{E}$ is an elliptic curve of good reduction over an $W(\fpbar)[1/p]$-algebra is represented by the (noncuspidal part) of the rigid-analytic  modular curve $X_1(N)$.  Let $X_1(N)_{\fpbar}$ be the corresponding modular curve over $\fpbar$. Let $\spe:X_1(N) \arr X_1(N)_{\fpbar}$ denote the specialization map. Then, for any $\tilde{\uE}=(\tilde{E},\tilde{P})$ of reduction $\uE=(E,P)$, we have $D(\tilde{\uE})=D(\uE)=\spe^{-1}([\uE])$, where $[\uE]$ denotes the isomorphism class of $\uE$ viewed as a point on $X_1(N)_{\fpbar}$.

When $N<4$, there may be points $\uE=(E,P)$ with a nontrivial automorphism group and the moduli space is no longer representable. We could still define  coarse moduli spaces $X_1(N)$, $X_1(N)_{\fpbar}$, and a specialization map $\spe: X_1(N) \arr X_1(N)_{\fpbar}$. In this case, we have $\spe^{-1}([\uE])=D(\uE)/\Aut(\uE)$.

We are interested in the $p$-adic dynamics of $T_\ell$ on the ordinary good reduction part of $X_1(N)$, or, to be more precise, on the set of $\CC_p$-points of $\sqcup_{\uE}\ D(\uE)$,
where $\uE$ runs over the isomorphism classes of ordinary elliptic curves over $\fpbar$ equipped with a point of order  $N$. If $n>0$ is an integer, the correspondence $T_\ell^n$ acts on $X_1(N)$ via the formula
\[
T_{\ell}^n([(\tilde{E},\tilde{P})])=\sum_C \mu_C [(\tilde{E}/C,P\ {\rm mod}\ C)],
\]
where $C$ runs over all  subgroups of $\tilde{E}$ of order $\ell^n$,  $[(\tilde{E},\tilde{P})]$ denotes the isomorphism class of $(\tilde{E},\tilde{P})$, and $\mu_C$ are positive multiplicities that can be explicitly calculated.

\subsection{Automorphisms of residue discs}\label{section:automorphisms} 
Let $\alpha(N)$ be the order of $\ell$ in $(\ZZ/N\ZZ)^\times$. Recall our standing assumption $(\star)$.
\begin{lem}\label{lem: branches} Let $\lambda\colon \uE_1=(E_1,P_1) \arr (E_2,P_2)=\uE_2$ be an $\ell^n$-isogeny of elliptic curves over $\fpbar$ sending $P_1$ to $P_2$. There is a unique rigid-analytic isomorphism, 
\[
\lambda^\ast\colon D(\uE_1) \arr D(\uE_2),
\]
such that for all $x=[(\tilde{E}_1,\tilde{P}_1)]  \in D(\uE_1)$, we have $\lambda^\ast(x)=[(\tilde{E}_1/C, \tilde{P}_1\ {\rm mod}\ C)]$, where $C$ is the unique subgroup of $\tilde{E}_1$ lifting $\Ker(\lambda)$. The inverse of $\lambda^\ast$ is given by $(\deg(\lambda)^{\alpha(N)-1}\lambda^\vee)^\ast$. If $\uE_1=\uE_2$ and $\lambda \in \ZZ$ then $\lambda^\ast={\rm id}$. For $\lambda_1\colon\uE_1 \arr \uE_2$, and $\lambda_2\colon\uE_2 \arr \uE_3$ we have $(\lambda_2\circ\lambda_1)^\ast=\lambda_2^\ast\circ\lambda_1^\ast$.
\end{lem}

\begin{proof} We claim there is a unique family of subgroup schemes $\tilde{C}$ of order $\ell^n$ of the universal family of elliptic curves over $D(\uE_1)$ with the following property: its specialization  at every point $(\tilde{E}_1,\tilde{P}_1)\in D(\uE_1)$ is a subgroup of $\tilde{E}_1$ reducing to $\Ker(\lambda)$ in $E_1$.  The morphism $\lambda^\ast$ is then defined via dividing the universal family of elliptic curves on $D(\uE_1)$ by the subgroup $\tilde{C}$. More precisely, let $(\calE_1,\calP_1)$ be a formal deformation of $(E_1,P_1)$ to a local artinian ring $(R,\germ_R)$ with a given isomorphism $R/\germ_R \cong \fpbar$. This data includes an isomorphism $\epsilon: \calE_1 \otimes R/\germ_R  \overset{\sim}{\arr} E_1$ sending $\calP_1$ to $P_1$. Since $(\ell,p)=1$, $\Ker(\lambda)$ is an \'etale subgroup scheme of $E_1$, and hence it lifts uniquely to a subgroup $\calC$ of the deformation $\calE_1$. We define 
\[
\lambda^\ast(\calE_1,\calP_1)=(\calE_1/\calC,\calP_1\ {\rm mod}\ \calC),
\] 
where the right side is viewed as a formal deformation of $(E_2,P_2)$ via the isomorphism
\[
(\calE_1/\calC) \otimes R/\germ_R \overset{\epsilon}{\arr} E_1/\Ker(\lambda) \arr E_2,
\]
where $E_1/\Ker(\lambda) \arr E_2$ is the isomorphism induced by $\lambda$. This defines a morphism of the formal deformation spaces $\tilde\lambda^\ast:\tilde{D}(\uE_1) \arr \tilde{D}(\uE_2)$. The desired map is the rigid-analytic morphism $\lambda^\ast$ associated to $\tilde\lambda^\ast$. The uniqueness follows from the uniqueness of the lift of $\Ker(\lambda)$. The other statements follow easily from the construction and its uniqueness.
\end{proof}

Let $\uE=(E,P)$ be an ordinary elliptic curve over $\fpbar$ along with a point of order  $N$. Let $\End_{\ell^\infty} (\uE)$ denote the monoid of $\ell$-power isogenies of $E$ fixing $P$. 
The above construction gives a monoid homomorphism 
\[
\ast: \End_{\ell^\infty} (\uE)  \arr \Aut(D(\uE)).
\]
sending $\lambda$ to $\lambda^\ast$. 
\begin{prop}
The kernel of $\ast$ is precisely $\ZZ \cap  \End_{\ell^\infty} (\uE)$.
\end{prop}
\begin{proof} Let $\lambda \in  \End_{\ell^\infty} (\uE)$ such that $\lambda^\ast = 1$. From the construction of $\lambda^\ast$ that means that for any lift $\tilde E$ with its unique subgroup $\tilde C$ lifting $C:= \Ker(\lambda)$, we have $\tilde E/C \cong \tilde E$ and so that $\tilde E$ has an endomorphism of degree $\sharp C$. As most lifts have only $\ZZ$ as endomorphisms, for some lift $\tilde E$ we most have $\tilde C = \Ker(\ell^n)$ and so $\lambda = \pm \ell^n$. The converse proceeds the same way. 
\end{proof}

Using the above construction, we have a homomorphism of monoids (that we still denote $\ast$ hoping no confusion will arise)
\[
\ast: \Omega(\Lambda_\ell(\uE,N),\uE) \arr \Aut(D(\uE)),
\]
sending $\omega$ to $\omega^*:=\tilde{\omega}^*$, where $\tilde{\omega}$ is the isogeny associated to $\omega$. Note that even when the isogeny $\tilde{\omega}$ may only be defined up to sign, $\tilde{\omega}^*$ is uniquely defined.  More generally, if $\omega$ is a walk in $\Lambda(\uE,N)$ of length $d$ from $[\uE_0]$ to $[\uE_d]$, then $\omega^\ast$ can be similarly defined to be an isomorphism
\begin{equation}\label{eqn: omega star}
\omega^\ast:D(\uE_0) \arr D(\uE_d).
\end{equation}
induced by the walk $\omega$. For two walks $\omega_1$, $\omega_2$ such that the endpoint of $\omega_1$ is the starting point of $\omega_2$, we have $(\omega_2 \circ \omega_1)^\ast=\omega_2^\ast \circ \omega_1^\ast$.

\begin{prop}\label{prop:cyclic} The image of $\ast: \Omega(\Lambda_\ell(\uE,N),\uE) \arr \Aut(D(\uE))$ is a cyclic group.
\end{prop}

\begin{proof} Note that this image is a subgroup of the image of $\ast: \Omega(\Lambda_\ell(E,1),[E]) \arr \Aut(D(E))$. So it is enough to prove the image of the latter is cyclic. Note that for a walk $\omega_1$ in $\Lambda_\ell(E,1)$, we have $(\omega_1^\vee)^\ast=(\omega_1^\ast)^{-1}$. Consequently $\omega^\ast$ depends only on the class of $\omega$ under the equivalence relation $\sim_h$. Thus, applying Corollary \ref{cor:walks}, it follows that the image of $\ast$ is generated by an automorphism conjugate to $(\omega_{rim}^+)^\ast$.
\end{proof}
Finally, we define
\[
f\colon \Omega(\Lambda_\ell(E,1),[E]) \arr \ZZ_p^\times, \quad
\omega \mapsto f_\omega,
\]
where $f_\omega$ is the image of $\tilde{\omega}$ under  $\End(E) \arr \End(T_pE) \cong \ZZ_p$. When $N\leq 2$, $f_\omega$ is defined only up to sign. We denote by $\overline{f}_\omega$ the image of $\overline{\tilde{\omega}}$ in $\ZZ_p^\times$, where the bar sign denotes complex conjugation in $\End(E)$.

 \subsection{$T_\ell$ via walks} Let $n>0$ be an integer. Let $\Walk(n,[\uE])$ be the set of all walks of length~$n$ in $\Lambda_\ell(\uE,N)$ starting at $\uE$. For $\omega \in  \Walk(n,[\uE])$, let $[\underline{E}_\omega]$ denote its endpoint.
The action of $T_\ell^n$ on the vertices of $\LambdaEN$ is then none other than $T_\ell^n([\uE])=\sum_{\omega \in \Walk(n,[\uE])} [\underline{E}_\omega]$.  

 \begin{prop}\label{prop:hecke factorization}
The correspondence $T_\ell^n$  on the residue disc $D(\uE)$ decomposes as a sum of rigid-analytic functions. More explicitly, 
\[T_\ell^n = \sum_{\omega\in \Walk(n,[\uE])}  \tilde\omega^\ast,\] 
where $\tilde \omega^\ast\colon D(\uE) \arr D(\uE_\omega)$ is the isomorphism defined in~(\ref{eqn: omega star}). 
 \end{prop}

\begin{proof}
Proceeding by induction on $n$, it suffices to check the case $n = 1$. In this case, for any point $(\tilde E, \tilde P)$ in $D(\uE)$ we have 
\[ T_\ell(\tilde E, \tilde P) = \sum_{C \subset \tilde E} (\tilde E/C, \tilde P {\rm \; mod\;} C),\]
where the summation is over all subgroups $C$ of $\tilde E$ of order $\ell$ (that correspond bijectively under reduction with the subgroups of $\uE$ of order $\ell$). On the other hand, the last sum is precisely $\sum_{\omega \in {\rm Walk}(1, [\uE])}\tilde \omega^\ast(\tilde E, \tilde P)$, by Lemma~\ref{lem: branches}.
\end{proof}

\section{Dynamics of $T_\ell$ on ordinary points}
\label{sec: dynamics on ordinary}

\id \subsection{}Proposition \ref{prop:hecke factorization} shows that the flow of the disc $D(\uE)$ under the iteration of $T_\ell$ is exactly like the flow of  $\uE\in \LambdaENv$ under the iteration of  $T_\ell$. The mass of $T_\ell^n(\uE)$ escapes to infinity. More precisely, we have the following proposition. 
\begin{prop} For $[\uE]\in \LambdaENv$ and a positive integer $n$, define the probability measure \[\delta_n = \delta_{n, [\uE]} = \frac{1}{\sharp {\rm Walk}(n, [\uE])} \sum_{\omega \in  {\rm Walk}(n, [\uE])} \delta_{[\uE_\omega]},\] where $\delta_{[\uE_\omega]}$ is the Dirac measure on $\LambdaENv$ concentrated at $[\uE_\omega]$. Then, for any compactly supported function $f\colon \LambdaENv \arr \CC$, we have $\int f d\delta_n  \arr 0 $ as $n \arr \infty$.
\end{prop}
\begin{proof} As $\pi_N$ is a covering map that commutes with $T_\ell$, it is enough to prove the statement for $N=1$. In this case, we may further reduce the question to the $\ell$-volcano $\Lambda_\ell(E, 1)$, where it follows from standard techniques of random walks; see, for example, \cite{Saw, Woe}.
\end{proof}

Therefore, the further study of the dynamics of $T_\ell$ on $D(\uE)$ is reduced to the study of the dynamics of branches of  $T_\ell^n$ that bring the disc back to itself, or, equivalently, the branches corresponding to the closed walks of $\LambdaEN$. Let $\omega$ be a closed walk starting and ending at $[\uE]$.  Let  $\bar{\omega}$ denote its image in $\Lambda_\ell(E,1)$.
It is clear from the definitions that we have a commutative diagram 
\[ \xymatrix{
D(\uE) \ar[d]^{\pi_N}_\cong\ar[r]^{\omega^\ast} & D(\uE) \ar[d]^{\pi_N}_\cong \\
D(E) \ar[r]^{\bar{\omega}^\ast} & D(E)
}\]
This shows that the dynamics of $\omega^\ast$ on $D(\uE)$ is equivalent to the dynamics of $\bar{\omega}^\ast$ on $D(E)$. In view of Proposition \ref{prop:cyclic}, understanding the dynamics of such functions  boils down to understanding the dynamics of a single function on the open unit disc; that of the automorphism of $D(E)$ induced by $\omega_{rim}^+$. In the following, for a closed walk $\omega$ in $\Lambda_\ell(E,1)$, we explicitly calculate $\omega^\ast$ on $D(E)$ in terms of a coordinate on it.

\begin{thm}\label{thm: dynamics on the disc} There is a choice of local parameter $t$ on $D(E)$ such that for any $\omega \in \Omega(\Lambda_\ell(E,1),[E])$, the automorphism $\omega^\ast: D(E) \arr D(E)$ is given by the map
\[
t \mapsto (1+t)^{\overline{f}_\omega/{f_\omega}}-1
\]
where $f_\omega \in \ZZ_p^\times$ is defined in \S\ref{section:automorphisms}. Such a map is an isometry of the disc.
\end{thm}

The theorem follows from the following Lemma.

\begin{lem}\label{lem:serre-tate}
Let $(\Abar,\mu_{\Abar})$ be a principally polarized ordinary abelian scheme of dimension $g$  over a finite field $\kappa$ of characteristic $p$. Fix an isomorphism $T_p\Abar\cong \ZZ_p^g$. Let 
\[
f: \Abar \arr \Abar
\]
 be an isogeny of degree $n$ prime to $p$, and $f^\dagger$ denote the image of $f$ under the Rosati involution induced by $\mu_{\Abar}$.  Let $M(f^{-1})=(f_{ij})_{i,j} \in \GL_g(\ZZ_p)$ (respectively, $M(f^\dagger)=(g_{ij})_{i,j})$ denote the matrix of the automorphism on  $T_p\Abar\cong \ZZ_p^g$ induced by $f^{-1}$ (respectively, $f^\dagger$). Let $f$ denote the automorphism of the deformation space $D$ of $(\Abar,\mu_{\Abar})$ induced by $f:\Abar \arr \Abar$. There is an isomorphism $D \cong \Spf[[t_{ij}]]_{i,j}$  under which the automorphism $f$ takes the shape
\[
\displaystyle{f^*(t_{ij})=\prod_{1 \leq r,s \leq g} (1+t_{rs})^{   f_{ri}g_{sj} }-1 \qquad \qquad 1 \leq i,j \leq g}.
\]
\end{lem}
\begin{proof}  Let $\calC_{\kappa}$ denote the category of local Artin rings $(R,\germ_R)$ such that $R/\germ_R = \kappa$. By Serre-Tate's theory of deformation of ordinary abelian varieties, deformations of $\Abar$ over $R$ are in correspondence with the elements of $\Hom_{\ZZ_p}(T_p\overline{A} \otimes_{\ZZ_p} T_p\overline{A}^\vee,\widehat{\GG}_m(R))$, where $\widehat{\GG}_m$ denotes the multiplicative formal group over $R$. We briefly recall (following \cite{Katz}) how a bilinear form $\eta_A\in \Hom_{\ZZ_p}(T_p\Abar \otimes_{\ZZ_p} T_p\Abar^\vee,\widehat{\GG}_m(R))$ is associated to a deformation $A$ of $\Abar$ over $R \in \calC_{\kappa}$. By Serre-Tate's theory of deformation of abelian varieties \cite{Katz} (See also \cite[Appendix]{Messing}), to give a deformation $\Abar$ over $R$ is equivalent to give a deformation $A[p^\infty]$ of $\Abar[p^\infty]$ over $R$. This, in turn, is equivalent to giving an extension
\[
0 \arr \widehat{A}  \arr A[p^\infty] \arr T_p\Abar \otimes_{\ZZ_p} (\QQ_p/\ZZ_p) \arr 0,
\]
where $\widehat{A}$, the formal group of $A$, is isomorphic to the unique toroidal formal group lifting $\widehat{\Abar}$. Every such extension is the push-out of the canonical extension 
\[
0 \arr T_p\Abar \arr T_p\Abar \otimes_{\ZZ_p} \QQ_p \arr T_p\Abar \otimes_{\ZZ_p} (\QQ_p/\ZZ_p) \arr 0
 \]
 by a unique map 
 \[
 \phi_A:  T_p\Abar  \arr \widehat{A}(R)
 \]
 defined as follows. Choose $N$ such that $\germ_R^{N+1}=0$, and note that $\widehat{A}(R)$ is killed by $p^N$, as $p \in \germ_R$. Let $x=(x_n)$ be an element of $T_p\Abar$. Pick any $n \geq N$, and let $\tilde{x}_n$ be an arbitrary lift of $x_n\in \Abar(\overline{\FF}_p)$ to $A(R)$. We define $\phi_A(x)=p^n\tilde{x}_n$ which is easily seen to be independent of the choice of $n \geq N$. The definition is independent of the choice of $\tilde{x}_n$, as any two choices differ by an element on $\widehat{A}(R)$ which is killed by $p^N$. Also, $\phi_A(x)$ lies in $\widehat{A}(R)$ since its reduction mod $\germ_R$ is zero by construction. Finally, we use the Weil pairing to identify $\widehat{A}(R)$ with $\Hom(T_p{\Abar^\vee},\widehat{\GG}_m(R))$ (this identification is valid over $\kappa$, and can be lifted to $R$ since $\widehat{A}$ is of multiplicative type). This re-interprets $\phi_A:  T_p\Abar \arr \widehat{A}(R)$ as a bilinear form $\eta_A \in \Hom_{\ZZ_p}(T_p\Abar \otimes_{\ZZ_p} T_p\Abar^\vee,\widehat{\GG}_m(R))$.
 
  First we claim the following general statement. Let $\delta:A_1 \arr A_2$ be a prime-to-$p$ isogeny of principally polarized ordinary abelian schemes over $R \in \calC_\kappa$, and $\delta:\Abar_1 \arr \Abar_2$ be its reduction. Then the following diagram 
  \begin{equation*}\label{equation: bilinear forms}
   \xymatrix{ 
T_p\Abar_1 \otimes_{\ZZ_p} T_p{\Abar_1^\vee}  \ar[d]^{\delta\otimes (\delta^\vee)^{-1}} \ar[r]^{\qquad \eta_{A_1}} & \widehat{\GG}_m(R)  \ar@{=}[d] 
\\
 T_p\Abar_2 \otimes_{\ZZ_p}  T_p{\Abar_2^\vee}  \ar[r]^{\qquad  \eta_{A_2}} & \widehat{\GG}_m(R) }
 \end{equation*}
 is commutative.  The commutativity follows from the commutativity of the following diagram (which itself follows from the construction):
  \begin{equation*}
\label{equation: phi} \xymatrix{ 
T_p\Abar_1  \ar[r]^{\phi_{A_1}} \ar[d]^{\delta}  & \widehat{A}_1(R) \ar[d]_{\delta}  \ar[r] & \Hom_{\ZZ_p}(T_p\Abar_1^\vee,\widehat{\GG}_m(R)) \ar[d]^{(\delta^\vee)^*}
\\
T_p\Abar_2 \ar[r]^{\phi_{A_2}} & \widehat{A}_2(R) \ar[r]& \Hom_{\ZZ_p}(T_p\Abar_2^\vee,\widehat{\GG}_m(R)) 
}
 \end{equation*}

 Going back to the proof, let $\Cbar$ denote the kernel of  the isogeny $f$, and $\pi: \Abar \arr \Abar/\Cbar$ be the canonical projection. There is a unique isomorphism $\lambda:  \Abar/\Cbar \xrightarrow{\sim} \Abar$ such that  $f=\lambda \circ \pi$. Let $(A,C)$ be a deformation of  $(\Abar,\Cbar)$ to $R \in \calC_{\kappa}$, which implies that $A/C$ is a deformation of $\Abar/\Cbar \cong \Abar$ to $R$. By the above discussion, to deformations $A,A/C$ of $\Abar$, we can associate bilinear forms 
 \[
 \eta_A \in \Hom_{\ZZ_p}(T_p\Abar \otimes_{\ZZ_p} T_p\Abar^\vee,\widehat{\GG}_m(R))
 \] 
 \[
 \eta_{A/C} \in \Hom( T_p(\Abar/\Cbar) \otimes_{\ZZ_p}  T_p((\Abar/\Cbar)^\vee),\widehat{\GG}_m(R)).
 \]
 Identifying $\Abar/\Cbar$ with $\Abar$ via $\lambda$, we obtain $\eta_{A/C}^\prime \in \Hom_{\ZZ_p}(T_p\Abar \otimes_{\ZZ_p} T_p\Abar^\vee,\widehat{\GG}_m(R))$ corresponding to $A/C$ in terms of $\Abar$.   In fact, $\eta_{A/C}^\prime:=\eta_{A/C} \circ \underline{\lambda}^{-1}$, where
 \[
 \underline{\lambda}:  T_p(\Abar/\Cbar) \otimes_{\ZZ_p}  T_p(\Abar/\Cbar)^\vee \arr T_p\Abar \otimes_{\ZZ_p} T_p{\Abar^\vee}  
 \]
 is given by $\lambda \otimes (\lambda^\vee)^{-1}$.  Similarly, define  $\underline{\pi}=\pi \otimes (\pi^\vee)^{-1}$. From the above discussion, it follows that the following diagram is commutative:

 \begin{equation*}
\label{equation: bilinear forms2} \xymatrix{ 
T_p\Abar \otimes_{\ZZ_p} T_p{\Abar^\vee}   \ar[r]^{\qquad \eta_{A/C}^\prime} & \widehat{\GG}_m(R)  \ar@{=}[d] 
\\
 T_p(\Abar/\Cbar) \otimes_{\ZZ_p}  T_p((\Abar/\Cbar)^\vee) \ar[u]^{\underline{\lambda}} \ar[r]^{\qquad   \qquad   \eta_{A/C}} & \widehat{\GG}_m(R)  \ar@{=}[d] 
  \\
T_p\Abar \otimes_{\ZZ_p} T_p{\Abar^\vee} \ar[r]^{\qquad \eta_A} \ar[u]^{\underline{\pi}}  & \widehat{\GG}_m(R) }\end{equation*}

Let $\{v_i\}_{i}$ be the standard basis for $T_p\Abar$ under our fixed isomorphism $T_p\Abar \cong \ZZ_p^g$, and let $\{w_i=\mu_{\Abar}(v_i)\}_i$ be the corresponding basis for $T_p\Abar^\vee$. Let $M(f^{-1})=(f_{ij}) \in \GL_g(\ZZ_p)$ be the matrix of  inverse of the map induced by $f$ on $T_p\Abar$ in this basis. Let $f^\dagger=\mu_{\Abar}^{-1} \bar{f}\mu_{\Abar}$ denote the image of $f$ under the Rosati involution associated to $\mu_{\Abar}$. Let $M^\vee(f^\vee)$ denote the matrix, with respect to $\{w_i\}$, of the map induced map by $f^\vee$ on $T_p\Abar^\vee$. This matrix can be calculated as follows. We have
\[
f^\vee(w_i)=\mu_{\Abar}(\mu_{\Abar}^{-1}f^\vee\mu_{\Abar})(v_i)=\mu_{\Abar}(f^\dagger(v_i)),
\]
which shows that $M^\vee(f^\vee)=M(f^\dagger)=(g_{ij})^{-1}$.  There is an isomorphism $D \cong \Spf[[t_{ij}]]$, where for any deformation $A$ over any $R \in \calC_\kappa$ we have $1+t_{ij}(A)=\eta_A(v_i\otimes w_j)$. To calculate $f^*(t_{ij})$, we use the above diagram to write
\[
\eta_A^\prime(v_i \otimes w_j)=\eta_A(f^{-1}(v_i)\otimes f^\vee(w_j))=\eta_A( \sum_r f_{ri}v_r \otimes \sum_s g_{sj}w_s)=\prod_{r,s}
\eta_A(v_r\otimes w_s)^ {f_{ri}g_{sj}}
 \]
 which proves the statement of the lemma.
 \end{proof}

\subsection{Periodic points: canonical and quasi-canonical lifts} Let $E$ be an ordinary elliptic curve over $\fpbar$. Then $\End(E)$ is isomorphic to  an order of a prime-to-$p$ conductor $\gerf(E)$ in a quadratic imaginary field $K=\QQ(\sqrt{d})$, $d$ a square-free negative integer. There is a unique (up to isomorphism)  elliptic curve $\tilde{E}_{can}$ over $W(\fpbar)$ lifting $E$ such that the natural reduction map  
\[
\End(\tilde{E}_{can}) \arr \End(E)
\]
 is an isomorphism. This elliptic curve is called the \textit{canonical lift} of $E$. A lift $\tilde{E}$ of $E/\fpbar$ is  \textit{quasi-canonical} if the image of the reduction map $\End(\tilde{E}) \arr \End(E)$ is an order in $\End(E)$; its conductor is necessarily of the form $p^a\gerf(E)$ for some $a \geq 0$. The case $a=0$ corresponds to the canonical lift.

Let $\uE=(E,P)$ consist of an ordinary elliptic curve $E$ over $\fpbar$ and a point $P$ of order  $N$ on~$E$. We say $\tilde{\uE}=(\tilde{E},\tilde{P})$ is a (quasi-) canonical lift of $\uE$ if $\tilde{E}$ is a (quasi-) canonical lift of $E$. Since $(p,N)=1$, the point $P$ can be uniquely lifted, and hence there is a unique canonical lift of $\uE$ which we denote by~$\tilde{\uE}_{can}$.   We write $\tilde{D}(\uE)\cong \Spf(W(\fpbar)[[t]])$, where $t$ is the Serre-Tate coordinate on $\tilde{D}(\uE)$. We prove the following.

\begin{prop}\label{prop: qc} Let notation be as above.

\begin{enumerate}
\item We have $t(\tilde{\uE}_{can})=0$.
\item $\tilde{\uE}$ is a quasi-canonical lift of $\uE$ of conductor $p^a\gerf(E)$ if and only if  $1+t(\tilde{\uE})$ is a primitive $p^a$-th root of unity.
\end{enumerate}
In other words, $D(\uE)$ is isomorphic to the open unit disc around $1$, the canonical lift corresponds to $1$, and the other quasi-canonical lifts correspond to the nontrivial $p$-power roots of unity.
\end{prop}

\begin{proof}  We will use the Serre-Tate theory of local moduli of ordinary abelian varieties. We refer to the proof of Lemma \ref{lem:serre-tate} for a summary of their construction and for notation. First we recall the following fact (see \cite{Katz}). Let $A_1$ and $A_2$ be abelian varieties over $\fpbar$, and $\tilde{A}_1,\tilde{A}_2$, respective lifts corresponding to pairings $\eta_{\tilde{A}_1},\eta_{\tilde{A}_2}$. Let $\delta\colon A_1 \arr A_2$ be a morphism. Then $\delta$ lifts to a morphism $\tilde{\delta}\colon \tilde{A}_1 \arr \tilde{A}_2$ if and only if  for all $v \in T_p A_1, w \in T_pA^\vee_2$ we have 
\[
\eta_{\tilde{A}_1}(\delta(v),w)=\eta_{\tilde{A}_2}(v,\delta^\vee(w)).
\] 
We will apply this with $A_1=A_2=E$. Consider $E$ with its canonical polarization and hence a canonical isomorphism $T_pE \overset{\sim}\arr T_pE^\vee$. Under this identification, for any $\delta\in \End(E)$, the dual isogeny $\delta^\vee\!\! \in \End(E^\vee)$ is identified with $\bar{\delta} \in \End(E)$. Let $v$ be a basis element in $T_p E$. This provides us with the Serre-Tate  isomorphism 
$$\tilde{D}(\uE)\cong \tilde{D}(E)\cong \Spf(W(\fpbar)[[t]]),$$ 
where for any lift $(\tilde{\uE})$ of $(\uE)$, we have $1+t(\tilde{\uE})=\eta_{\tilde{E}}(v,v)$.  

In our notation, $\End^0(E) = K = \QQ(\sqrt{d})$, $d$ square-free. If $d \equiv 2, 3 \pmod{4}$ let $\gamma = \sqrt{d}$, and if $d \equiv 1 \pmod{4}$ let $\gamma = (1+\sqrt{d})/2$. Thus, $\calO_K=\ZZ+\gamma\ZZ$.
The lift $\tilde{\uE}$ is a quasi-canonical lift of $\uE$ of conductor dividing $p^a\gerf(E)$ in $K$ if and only if $\alpha=p^a\gerf(E)\gamma$ can be lifted from $\End(E)$ to $\End(\tilde{E})$. This is equivalent to having $\eta_{\tilde{E}}( \alpha v,v)=\eta_{\tilde{E}}(v,\overline{\alpha}v)=\eta_{\tilde{E}}(\overline{\alpha}v,v)$,
which is equivalent to \[\eta_{\tilde{E}}(v,v)^{\alpha -\overline{\alpha }}=(1+t(\tilde{\uE}))^{\alpha -\overline{\alpha }}=1.\] 
If $d \equiv 2, 3 \pmod{4}$, we have $(1+t(\tilde{\uE}))^{\alpha -\overline{\alpha }} = (1+t(\tilde{\uE}))^{2p^a\gerf(E)\sqrt{d}}$; if $d \equiv 1 \pmod{4}$, we have $(1+t(\tilde{\uE}))^{\alpha -\overline{\alpha }} = (1+t(\tilde{\uE}))^{p^a\gerf(E)\sqrt{d}}$.

If $p \neq 2$, taking into account that $p$ splits in $K$, we find that $2\gerf(E)\sqrt{4d}$, respectively $\gerf(E)\sqrt{d}$, are $p$-adic units, and so $(1+t(\tilde{\uE}))^{\alpha -\overline{\alpha }} = 1$ if and only if $(1+t(\tilde{\uE}))^{p^a} = 1$.
If $p = 2$, then since $p$ splits in $K$, we must have $d \equiv 1 \pmod{4}$ and, again, $(1+t(\tilde{\uE}))^{\alpha -\overline{\alpha }} = 1$ if and only if $(1+t(\tilde{\uE}))^{p^a} = 1$.

This proves the second part of the Proposition. For the first part, simply put  $a=0$, to obtain $1+t(\tilde{E}_{can})=1$.
\end{proof}
 Let $\ord = \ord_p$ denote the $p$-adic valuation on $\bar\QQ_p$, normalized so that $\ord(p) = 1$.
\begin{prop}\label{prop: 1234}   Let $\omega \in \Omega(\Lambda_\ell(\uE,N),[\uE])$ be such that $\omega^\ast\in \Aut(D(\uE))$ is nontrivial. Consider the dynamics of  $\{(\omega^\ast)^n\}_{n\geq 1}$ on $D(\uE)$. The following hold:
\begin{enumerate}
\item Every pre-periodic point in $D(\uE)$ is periodic.
\item Let $m>0$ be an integer. The index $[\calO_K: \ZZ[\tilde{\omega}^m]] < \infty$ unless, possibly, when $K = \QQ(\sqrt{-\ell})$. If $K = \QQ(\sqrt{-\ell})$ then $\tilde{\omega} = \pm (-\ell)^{b/2}$ for some odd positive integer $b$, and $[\calO_K: \ZZ[\tilde{\omega}^m]] = \infty$ if and only if $m$ is even. 
\item Assume that $[\calO_K: \ZZ[\tilde{\omega}^m]] < \infty$, then periodic points in $D(\uE)$ of order dividing $m$  are exactly the  quasi-canonical points of $D(\uE)$ of conductor dividing $[\calO_K:\ZZ[\tilde{\omega}^m]]$. In terms of the Serre-Tate coordinate $t$, these correspond to points $\tilde{\uE}\in D(\uE)$ for which $(1+t(\tilde{\uE}))^{p^a}=1$, where $a=\ord_p([\End(E):\ZZ[\tilde{\omega}^m]])$. 
\item If $[\calO_K: \ZZ[\tilde{\omega}^m]] = \infty$ any point is a periodic point of order dividing $m$.
\end{enumerate}
\end{prop}

\begin{proof}  Pre-periodic points are periodic simply because $\omega^\ast \in \Aut(D(\uE))$. To prove claim (2) note that $\tilde{\omega}$ is a non-integer integral element of $K$ (thus generating $K$ over $\QQ$) whose norm is $\ell^n$ for some $n>0$. If $[\calO_K: \ZZ[\tilde{\omega}^m]] = \infty$ then $\tilde{\omega}^m = \pm \ell^{mn/2}$. This implies that the discriminant of $K$ divides $-4\ell^n$ and so the discriminant of $K$, which is not $-3$ or $-4$, must be equal to $-\ell$ or $-4\ell$ if~$\ell$ is odd, and $-8$ if $\ell = 2$. Thus, $K =  
\QQ(\sqrt{-\ell})$. The rest of (2) is clear.

Let $\tilde{\uE}=(\tilde{E},\tilde{P}) \in D(\uE)$. It follows from the construction of $\omega^\ast$ that $(\omega^\ast)^m(\tilde{\uE})=\tilde{\uE}$ if and only if $\tilde{\omega}^m\in \End(E)$ can be lifted to an endomorphism of $\tilde{E}$. If $\tilde{\omega}^m$ is not an integer, this is equivalent to $\tilde{\uE}$ being a quasi-canonical lift with endomorphism ring containing $\ZZ[\tilde{\omega}^m]\subset \End(E)$, i.e.,~$\tilde{\uE}$ being a quasi-canonical lift of conductor dividing $[\calO_K:\ZZ[\tilde{\omega}^m]]$. In terms of the Serre-Tate coordinate, we have already seen in the proof of Proposition \ref{prop: qc} that $(1+t(\tilde{\uE}))^{p^a}=1$ if and only if~$\tilde{E}$ is a quasi-canonical lift of conductor dividing $p^a\gerf(E)$. If $\tilde{\omega}^m$ is an integer then every point is $m$-periodic. 
The rest follows immediately.
\end{proof}

\begin{cor}\label{cor: summary of T ell orbit} The $T_\ell$-forward orbit of any CM point $\tilde\uE$ with ordinary reduction and conductor $p^a\mathfrak{f}(\tilde \uE)$ is a discrete set, intersecting each residue disc at a finite set contained in the $p^a$-roots of unity (relative to the Serre-Tate coordinate). The prime-to-$\ell$ part of the conductor of the endomorphism ring in such an orbit is an invariant. 
\end{cor}

\subsection{Closure of orbits and CM} 
\begin{prop}\label{prop: 1235} Let $\uE = (E, P)$ be an elliptic curve of ordinary reduction over $\CC_p$. Let $K = \End^0(E)$.
\begin{enumerate}
\item If $K$ is a CM field, the orbit of $\uE$ under $T_\ell$ is a discrete set consisting only of elliptic curves with the same endomorphism algebra $K$. Consequently, any point in the closure of the orbit of $\uE$ has CM by an order whose conductor differs from that of $\End(E)$ only at $\ell$.
\item If $K = \QQ$, no point in the closure of the orbit of $\uE$ under $T_\ell$ has CM.
\end{enumerate}
\end{prop}
\begin{proof} Corollary~\ref{cor: summary of T ell orbit} contains (1), since $p$-power roots of unity form a discrete set of the unit disc. Consider then the case $K = \QQ$. Suppose that there is a CM point $\uE_0$ in the closure of the orbit of $\uE$. Let $F = \End^0(E_0)$ and let $p^a$ be the $p$-part of the conductor of $\End(E_0)$ in $F$. 

We claim that $F = \End^0(\overline E)$ where $\overline E$ is the reduction of $E$ modulo the maximal ideal of $\CC_p$. Indeed, since $\uE_0$ is in the closure of the orbit of $\uE$ there is an $\ell$-isogeny, $\uE \arr\uE_1$, where $\uE_1$ is in the residue disc $D(\uE_0)$, and so $F = \End^0(\overline{E}_0) = \End^0(\overline{E}_1) = \End^0(\overline{E})$.

Let $S$ be the set of all CM points that are quasi-canonical lifts of $\overline{E}$, with conductor whose $p$-part is $p^a$. They and $\uE$ belong to the same residue disc $D(\uE)$. Let 
\[{\rm dist}(\uE, S) = \{ {\rm dist}(\uE, s): s\in S\},\] 
where we view here the residue disc as a subset of $\CC_p$ by means of a Serre-Tate coordinate, $\uE$ and $s$ are elements of $\CC_p$, and the distance is the $p$-adic metric. Note that ${\rm dist}(S)$ is a finite set of \textit{positive} real numbers. 

Now, let $f$ be an $\ell$-isogeny $\uE \arr \uE^\prime$ such that $\uE^\prime \in D(\uE_0)$. The induced isomorphism of residue discs $f^\ast\colon D(\uE) \arr D(\uE_0)$ is an \textit{isometry}.\footnote{Any automorphism of the open unit disc defined over a discretely valued field extension of $\QQ_p$ is an isometry.} Thus,  ${\rm dist}(\uE, S) = {\rm dist}(\uE^\prime, f^\ast(S))$. Moreover, by a counting argument, $f^\ast(S)$ consists of \textit{all} quasi-canonical lifts in $D(\uE_0)$ with $p$ part of the conductor being $p^a$. Thus, $\uE_0 \in f^\ast(S)$ and we find that ${\rm dist}(\uE^\prime, \uE_0)$ is bounded from below regardless of $f$. This is a contradiction. 
\end{proof}

\subsection{Dynamics of $t \mapsto (1+t)^\lambda - 1$.}
As we have seen, the action of the Hecke operator $T_\ell$ on points of ordinary reduction reduces, essentially, to the study of an automorphism of the open disc of radius~$1$, which is of the form $t \mapsto \omega^\ast(t) = (1+t)^{\bar f_\omega/f_\omega} - 1$. Therefore, we now direct our attention to automorphisms of the disc of the form 
\[ t \mapsto h(t) := (1+t)^\lambda - 1,\]
where $\lambda \in \ZZ_p^\times$. Recall that
$(1+t)^\lambda=\sum_{n\geq 0} {\lambda \choose n} t^n,$
where ${\lambda \choose n}=\frac{\lambda(\lambda-1)\cdots(\lambda-n+1)}{n!}$.

The nature of the orbit $\{h^n(x): n \in \ZZ\}$ is best understood through a sequence of reductions that while providing a complete picture also convince the reader that formulating a quantitive answer in the general case is too messy. 

Thus, let $F$ be a finite extension of $\QQ_p$ with a residue field of degree $p^f$ and ramification index~$e$. Let $\germ_F$ be the maximal ideal of $F$. Define for $n\geq 1$,
\[ U_n = 1 + \germ_F^n \subset \calO_F^\times.\]
We have a decomposition:
\[ \calO_F^\times = \mu_{p^f - 1} \times U_1.\]
We change coordinate so that the centre is $1$, and so our map in the new coordinate (still called~$t$) is simply 
\[ t \mapsto h(t) = t^\lambda.\]
Let $x\in \calO_F^\times$; since $x$ should reduce to $1$, we take $x\in U_1$.  The group $U_1$ has a non-canonical decomposition
\[ U_1 = \mu_{p^a} \times U^\prime, \]
where $U^\prime$ is isomorphic as a $\ZZ_p$-module to $\ZZ_p^{g}$, $g = [F:\QQ_p]$. The action of $h$ respects this decomposition and we see that at the expense of passing from $\lambda$ to $\lambda^{p^a}$ (which will result in decomposing the orbit into a union of finitely many translated-orbits for $\lambda^{p^a}$),
\begin{itemize}
\item \textit{We may reduce to the case $x\in U^\prime$.}
\end{itemize}
Choose an $n_0> \frac{e}{p-1}$ and a sufficiently large $n_1$. The homomorphism
\[ U^\prime \arr U_{n_0}, \qquad x \mapsto x^{p^{n_1}}, \]
is a homeomorphism from $U^\prime$ into an open subgroup of finite index of $U_{n_0}$ that commutes with~$h$. To stress, it is injective. Thus, we accept the next reduction:
\begin{itemize}
\item \textit{We may reduce to the case $x\in U_{n_0}$.}
\end{itemize}
Now, the logarithm function, $\log(t)=\sum_{n\geq 1} (-1)^{n-1}\frac{t^n}{n}$, converges on $U_1$. If $n>\frac{e}{p-1}$, then 
\[
\log: U_n \overset{\sim}\Arr \germ_F^n
\] 
is an isometry whose inverse is given by the exponential function $\exp(t)=\sum_{n \geq 1} \frac{t^n}{n!}$.  That explains our interest in assuming that $x\in U_{n_0}$. Thus, by applying the logarithm, the orbit $\{x^{\lambda^n}: n\in \ZZ\}$ is transformed into the orbit $\{\log(x)\cdot \lambda^n: n\in \ZZ\} = \log(x)\cdot \{ \lambda^n: n\in \ZZ\}$. And so we accept the next reduction. 
\begin{itemize}
\item \textit{The essential information about the orbit of $h(x)$ is contained in the properties of the subset of $\ZZ_p^\times$ given by $ \{ \lambda^n: n\in \ZZ\}$.}
\end{itemize}
This last question is uniform. It doesn't know about $x$ or $F$. Unfortunately, also here we run into similar book-keeping issues, and so once again we satisfy ourselves with a reduction that comes from the decomposition $\ZZ_p^\times  = \mu_{p-1} \times (1+p\ZZ_p)$. 
\begin{itemize}
\item \textit{We may reduce to the case $\lambda\in (1+p\ZZ_p)$.}
\end{itemize}
Finally, by applying the logarithm, we find that
\begin{itemize}
\item \textit{The set $\{\lambda^n: n \in \ZZ\}$ in $1+p\ZZ_p$  is mapped bijectively to the set $\log(\lambda) \cdot \ZZ$ in $p\ZZ_p$, whose closure is a disc in $p\ZZ_p$ of radius $p^{-\ord(\lambda-1)}$, centred at $\log(\lambda)$.}
\end{itemize}

\

\id Without getting bogged down in the details, we can conclude the following:
\begin{prop}\label{prop: 1236}
Assume that $x$ is not a root of unity.
\begin{enumerate} 
\item The closure of the orbit of any point $x$ under $\omega^\ast$ is equal to a finite disjoint union of sets, each homeomorphic to $\ZZ_p$.
\item In contrast to the case of a CM point, where the orbit is countable and discrete, the closure of the orbit of $x$ under iterations of $T_\ell$ is of cardinality $2^{\aleph_0}$. 
\end{enumerate}
\end{prop}

\

\section{The $\ell$-isogeny graph: the supersingular case}

\subsection{Some notation}\label{subsec: notation for 5} Recall our choice of a fixed prime number $p$. In the following, $E, E_1,$ and so on, will typically denote supersingular elliptic curves in characteristic $p$. We denote ${\rm Frob}_E: E \arr E^{(p)}$ the Frobenius isogeny; occasionally we denote it $\varphi$ for simplicity of notation. We denote $B=B_{p, \infty}$ the quaternion algebra over $\QQ$ ramified precisely at $p$ and $\infty$ and denote~$\scrB$ its $p$-adic completion $B_{p,\infty} \otimes_\QQ \QQ_p$; $\calR$ will denote a maximal order of $B_{p, \infty}$ and $\scrR$ its $p$-adic completion. The order $\calR$ is not unique but $\scrR$ is. Let $q = p^2$ and $\FF_q$ a field of cardinality $q$ in a fixed algebraic closure $\bar \FF_p$. Denote $\QQ_q = W(\FF_q)[p^{-1}]$. 
When we talk about $\overline \QQ_p$ we have in mind an algebraic closure containing $W(\fpbar)$ and $\CC_p$ denotes its completion.

\subsection{Definition of supersingular graphs}
\id As in the ordinary case, in addition to the fixed prime~$p$, we now fix a prime~$\ell$, and a positive integer $N$, such that $p \neq \ell$, $(N, p\ell) = 1$. Although some of the results below apply to the case $\ell = 2$, some of the literature we are using necessitates assuming $\ell$ is odd and so we assume that from this point on. Recall that $\alpha(N)$ denotes the order of $\ell$ in the group $\ZZ/N\ZZ^\times$.
\begin{dfn}\label{dfn: solid rigid} A pair $(N, p)$ is called \textit{rigid} if for every supersingular elliptic curve $E$, and $P\in E[N]$ of order  $N$, we have $\Aut(E, P) = \{ 1\}$. It is called \textit{solid} if $\Aut(E, P) \subseteq \{ \pm 1\}$. If $p$ is clear from the context, we shall simply say $N$ is rigid, or solid. 
\end{dfn}
%
The classification of automorphism groups of elliptic curves (see \cite[III, \S10]{Sil}) implies:

\begin{lem} \label{lem: rigid and solid}If $(N, p)$ is rigid it is solid. Further: 
\begin{enumerate} 
\item If $N =1$, then $(N, p)$ is never rigid and is solid iff $p \equiv 1 \pmod{12}$.
\item If $N = 2$, then $(N, p)$ is never rigid and is solid iff $p \equiv 1 \pmod{4}$.
\item If $N = 3$, then $(N, p)$ is solid iff $p \equiv 1 \pmod{3}$. If it is solid then it is rigid. 
\item If $N>3$, then $(N, p)$ is rigid. 
\end{enumerate}
\end{lem}

To define supersingular graphs in a very concrete way, we make use of the following.
\begin{lem} \label{lem: models of elliptic curves over fpsquare} Every supersingular elliptic curve has a model $E$ over $\FF_{p^2}$ such that ${\rm Frob}_E^2 = p$, unique up to $\FF_{p^2}$-isomorphisms.
\end{lem}
\begin{proof}\footnote{We have learned of this nice fact from a post of Bjorn Poonen on Mathoverflow.} Let $\varphi = {\rm Frob}_E$. This Lemma is well-known for $\varphi^2 = -p$; cf. Serre \cite[p. 284]{Serre}, \cite[Lemma 3.20]{Baker}. The variant $\varphi^2 = p$ is obtained for $p>2$ by a quadratic twist. Even for $p=2$ we can pass from $\varphi^2  = -p$ to $\varphi^2 = p$, but this time by twisting by the element of $H^1(\Gal_{\fpbar/\FF_q}, \Aut(E))$ that sends Frobenius to $i$, an automorphism of order $4$ of $E$. 
\end{proof}

We fix such elliptic curves as representatives for the $\fpbar$-isomorphism classes of supersingular elliptic curves. 
Given $N$, we extend this choice to a choice of representatives $(E, P)$ for the $\fpbar$-isomorphism classes of pairs $(E_1, P_1)$, where $E_1$ is supersingular and $P_1$ is a point of $E_1$ of order ~$N$. There is no canonical way to choose the points $P$, and so we make an arbitrary choice once and for all, and call the set of representatives $Rep(N)$. For a representative $E$, the number of representatives $(E, P) \in Rep(N)$ is the number of orbits of $\Aut(E)$ acting on the points of order~$N$ in $E$.

We now define the \textit{supersingular graph $\Lambda_\ell(N)$ of level $N$}. The set of  vertices $\Lambda_\ell(N)^V$ of $\Lambda_\ell(N)$ is the set \textit{Rep}($N$) we have fixed above. The construction is such that $\Lambda_\ell(N)$  will be a directed graph, whose arrows are labeled, and we describe its arrows in two equivalent ways. 
\begin{enumerate}
\item For every vertex $(E, P)$, and a cyclic subgroup $C$ of $E$ of order $\ell$, draw an  arrow from the vertex  $(E, P)$ to the unique representative $(E_1, P_1)$ of $(E/C, P\ {\rm mod}\, C)$. Let $\pi_C\colon E \arr E/C$ be the canonical map, and let $\alpha\colon  E/C \arr E_1$ be an isomorphism such that $\alpha(\pi_C(P)) = P_1$. Note the morphism $\lambda = \alpha \circ \pi_C\colon  (E, P) \arr (E_1, P_1)$. It is unique up to  composition with elements of $\Aut(E_1, P_1)$.  An arrow is marked with the set of {\it labels} 
\[
\Aut(E_1, P_1) \circ \lambda.
\] 
Every element of $\Aut(E_1, P_1) \circ \lambda$ will be called a label of the arrow.
\item Let $(E, P)$ and $(E_1, P_1)$ be representatives in $\textit{Rep}(N)$, i.e., vertices of $\Lambda_\ell(N)$. The set of degree $\ell$ isogenies $\lambda\colon  (E, P) \arr (E_1, P_1)$, if non-empty, has a natural action of $\Aut(E_1, P_1)$, and we associate to every orbit  $\Aut(E_1, P_1)\circ \lambda$ of this action an arrow from $(E, P)$ to $(E_1, P_1)$, together with the set of labels $\Aut(E_1, P_1)\circ \lambda$. 
\end{enumerate}

\vspace{0.5cm}

\id In our second definition, every arrow, by means of its set of labels $\Aut(E_1, P_1)\circ \lambda$ produces a unique subgroup $C$ of $E$ of order $\ell$, which is the kernel of any of the isogenies in the set $\Aut(E_1, P_1)\circ \lambda$. Note that there is a unique isomorphism $\alpha\colon  E/C \arr E_1$ making the diagram below commutative: 
\[ \xymatrix{E \ar[rr]^\lambda\ar[dr]^{\pi_C} && E_1\\ & E/C \ar[ur]^{\alpha} & }\]
It follows that the point $P_1 = \lambda(P)$ is equal to $\alpha (\pi_C(P))$ and so the two constructions of the arrows in the graph are the same. 

\begin{rmk} Note that there is no obvious way to make the graphs $\Lambda_\ell(N)$ undirected; the natural idea of identifying an isogeny $\lambda$ with its dual $\lambda^\vee$ fails in general because $\lambda^\vee$ and $\lambda^{-1}$ differ by multiplication by $\ell$, which  acts non-trivially on points of order $N$ (unless $\ell \equiv 1 \pmod{N}$). 

Even in the case $N=1$ there is a problem as the graphs are typically non-symmetric. The number of arrows from $E_1$ to $E_2$ is $\Aut(E_1)/\Aut(E_2)$ times the number of arrows from $E_2$ to $E_1$. The adjacency matrix of $\Lambda_\ell(1)$ is the $\ell$-th Brandt matrix. See \cite{GrossHeights}, in particular Proposition~2.7.
Finally, we remark that in general the graphs may contain loops and multiple arrows. 
\end{rmk}

\subsection{Properties of supersingular graphs} Although the graphs $\Lambda_\ell(N)$ need not be simple in general, we do have the following proposition.

\begin{prop} If $N> 4\ell^2$ the graph $\Lambda_\ell(N)$ is simple, i.e., contains no loops or multiple edges; moreover, every arrow has a unique label in this case. The girth of the graph $\Lambda_\ell(N)$ is at least $\log_\ell(N/4)$.
\end{prop}
\begin{proof} The statement about the girth implies the one about the loops. Suppose that the girth is~$r$. Then, there is a representative $(E, P)$ and a $\gamma \in \End(E)$ such that the $\deg(\gamma) = \ell^r$ and $\gamma(P) = P$. Therefore, $N \vert \sharp \Ker(\gamma - 1)= \deg(\gamma - 1)$. Thinking about $\gamma$ as a complex number of modulus $\ell^{r/2}$ we have $\deg(\gamma - 1)  = \vert \gamma - 1\vert^2  \leq (\vert \gamma \vert + 1)^2 \leq (2\vert \gamma \vert)^2 = 4\ell^r$, which gives $N \leq 4\ell^r$.

Suppose we have two distinct arrows $f, g\colon  (E_1, P_1) \arr (E_2, P_2)$. Then $g^{-1}\circ f$ is a rational endomorphism of $(E_1, P_1)$, and $\gamma: = \ell g^{-1}\circ f = g^\vee \circ f$ is an endomorphism of $E_1$ of degree $\ell^2$ such that $\gamma(P_1) = \ell P_1$. Consequently, $N \vert \sharp \Ker(\gamma - \ell) = \vert \gamma - \ell\vert^2 \leq 4\ell^2$. Thus, if $N> 4\ell^2$ there are no multiple edges, and by Lemma~\ref{lem: rigid and solid} no multiple labels. 
\end{proof}

We note that there is a natural map
\[ \pi_N: \Lambda_\ell(N) \arr \Lambda_\ell(1), \qquad (E, P) \mapsto E, \]
and taking an arrow with label $\lambda\colon (E_1, P_1) \arr (E_2, P_2)$ to the arrow  $\lambda\colon E_1 \arr E_2$.

\

\id The construction of the graphs $\Lambda_\ell(N)$ is motivated by the study of the Hecke operator $T_\ell$; the directed walks of length $r$ in $\Lambda_\ell(N)$ that begin at $(E, P)$ describe the image of the point $(E, P) \in X_1(N)(\fpbar)$ under $T_\ell^r$, including multiplicities. 
With this in mind, we make a further study of these graphs. We first prove a technical lemma about translating isogenies to walks. 
 
\begin{lem}\label{lemma:walkisogeny} Let $\uE=(E,P)$ be a vertex of $\Lambda_\ell(N)$. Let $\lambda: \uE \arr \uE'$ be an isogeny of degree $\ell^r$. Let $\uE'' \in Rep(N)$, and  $\epsilon:\uE' \arr \uE''$ be an isomorphism. Then there is a walk of length $r$ in $\Lambda_\ell(N)$,
\[
\uE=\uE_0 \overset{\lambda_1}{\arr} \uE_1 \overset{\lambda_2}{\arr} \cdots \overset{\lambda_{r-1}}{\arr} \uE_{r-1} \overset{\lambda_r}{\arr} \uE_r=\uE'', \
\]
such that each $\lambda_i$ is a label for the corresponding arrow, and such that $\epsilon^{-1}\lambda_{r}...\lambda_2\lambda_1=\lambda$.
\end{lem}
 
 \begin{proof} Let $C_r=\Ker(\lambda)$, and factorize $\lambda$ as $(E,P) \arr (E/C_r, P \;{\rm mod }\;C_r) \overset{\delta}\cong \uE'$. Let 
\[ 0 = C_0 \subsetneq C_1 \subsetneq C_2 \subsetneq \dots \subsetneq C_r = \Ker(\lambda) \]
be a filtration by subgroups with cyclic quotients of order $\ell$. For each $0<i<r$, let $\epsilon_i$ be an isomorphism
\[\epsilon_i\colon  (E/C_i, P \;{\rm mod }\;C_i) \cong \uE_i , \]
where $\uE_i \in \textit{Rep}(N)$, such that $\epsilon_0=id$, $\uE_r=\uE''$, and  $\epsilon_r=\epsilon\delta$.  Let $\rho_i\colon  E/C_{i-1} \arr E/C_i$ be the canonical morphism.  Since $C_r$ is the kernel of $\lambda$, we have  $\lambda=\delta\rho_n...\rho_1$. This yields a factorization
\[
\epsilon \lambda= (\epsilon_r\circ \rho_r\circ\epsilon_{r-1}^{-1})\circ \cdots \circ (\epsilon_2 \circ\rho_2\circ\epsilon_1^{-1}) \circ ( \epsilon_1 \circ\rho_1\circ \epsilon_0^{-1}).
 \]
Note that each $\lambda_i:=\epsilon_i\circ \rho_n\circ\epsilon_{i-1}^{-1}$ is a label for an arrow in the graph $\Lambda_\ell(N)$ from $\uE_{i-1}$ to $\uE_i$. This provides a walk of length $r$ starting at $\uE$,  ending at $\uE''$, such that $\epsilon^{-1}\lambda_{r}...\lambda_2\lambda_1=\lambda$.
 \end{proof}

\begin{thm} \label{thm: properties of the graphs}The graphs $\Lambda_\ell(N)$ have the following properties.
\begin{enumerate} \item $\Lambda_\ell(N)$ is a connected, directed graph of fixed out-degree $\ell+1$.
\item The graph $\Lambda_\ell(N)$ is not bipartite. 
\item If $N$ is rigid, every vertex in $\Lambda_\ell(N)$ has in-degree $\ell+1$ as well.
\end{enumerate}
\end{thm}
\begin{proof} The fact that $\Lambda_\ell(N)$ is directed and every vertex has out-degree $\ell+1$ follows from the construction and the description of arrows by means of cyclic subgroups of order $\ell$. 
 
Let $G$ be the algebraic group over $\QQ$ whose $\QQ$ points are $B_{p, \infty}^1 = \{ x\in B_{p, \infty}\colon  {\rm Nm} (x) = 1\}$; it is a form of $\SL_2$ and so $G$ is semi-simple, almost simple, and simply connected. Fix two pairs $(E, P_1), (E, P_2)$ in $\textit{Rep}(N)$. We identify $B_{p, \infty}$ with $\End^0(E)$ and that gives us a maximal order $\calR$ that corresponds to $\End(E)$. The group $G(\QQ)$ is then the group of rational endomorphisms of~$E$ of norm $1$. We choose a symplectic basis for the Tate modules $T_u(E)$ for every prime $u \neq p$. The action of $G(\ZZ_u)$ on $T_u(E)$ induces an isomorphism $G(\ZZ_u)\cong \SL_2(\ZZ_u)$.

Using strong approximation for $G$ relative to the place $\ell$, we can find an element $y$ of $G(\QQ)$ that is very close, except possibly at $\ell$, to the element
\[ x\in G(\AA^f_\QQ), \qquad x = (x_u)_u = \begin{cases} 1 & u \nmid  N, \\ x_u & u \vert N,\end{cases}\]
where the elements $x_u, u\vert N$, are chosen to be in $\SL_2(\ZZ_u)$ and such that $\prod_{u \mid N} x_u \pmod{N} \in \SL_2(\ZZ/N\ZZ)$ takes $P_1$ to $P_2$. 

The element $y$ is thus integral, except possibly at $\ell$, and ${\rm Nm}(y)=1$. Replacing $y$ by $\ell^{\alpha(N)m}y$ for a suitable positive integer $m$, we have proven the existence of an isogeny $f\colon  (E, P_1) \arr (E, P_2)$ of degree a power of $\ell$. Using Lemma~\ref{lemma:walkisogeny}, we can factor $f$ as a composition of cyclic isogenies of degree $\ell$ and we thus find that $(E, P_1)$ is connected to $(E, P_2)$ in $\Lambda_\ell(N)$. Combined with the well-known fact that $\Lambda_\ell(1)$ is connected (namely, between any two supersingular elliptic curves over $\fpbar$ there is some isogeny of degree a power of $\ell$ -- see below), one concludes that $\Lambda_\ell(N)$ is connected as well. In fact, we have proven that for all $E$, and for all $P_1,P_2$ of order  $N$ on $E$, there is a walk of even length from $(E,P_1)$ to $(E,P_2)$. 

We now prove that $\Lambda_\ell(N)$ is not bipartite.  We first consider the case $N=1$. It is enough to prove that there is a closed walk of odd length in $\Lambda(1)$. By Lemma \ref{lemma:walkisogeny}, this is equivalent to proving that there is an elliptic curve $E$ with an endomorphism of degree $\ell^r$, where $r$ is a positive odd integer.  Let $E$ be an arbitrary supersingular elliptic curve, and consider the quadratic form in $4$ variables and discriminant $p^2$,
\[
\deg\colon \End(E)\cong \calR \arr \ZZ.
\]
We need to prove that $\deg$ represents an integer of the form $\ell^r$, for some $r$ a positive odd integer. We show that it represents every large enough positive integer prime to $p$. 
We will use the fact that the Integral Hasse Principle for a quadratic form on a rank $4$ lattice holds for large enough integers that are relatively prime to the discriminant of the quadratic form of the lattice. Since $\ell$ is prime to $p$, our desired result follows if we prove that the quadratic form $\deg$ represents every integer locally. The quadratic form $\deg$ is positive definite at infinity. At a finite place $v \neq p$, we have an isomorphism $\calR \otimes_\ZZ \ZZ_v \cong M_2(\ZZ_v)$ via which  norm  corresponds to the determinant, and every $v$-adic integer can be represented as a determinant. At $p$ we have $\calR\otimes_\ZZ\ZZ_p \cong \scrR$ which has a model   
\[\left\{ \begin{pmatrix} a & pb^\sigma \\ b & a^\sigma
\end{pmatrix}: a, b \in W(\FF_{p^2})\right\},\]
where the degree map sends the above element to the determinant $aa^\sigma - pbb^\sigma$. By local class field theory, the norm map $W(\FF_{p^2})^\times \arr \ZZ_p^\times$ is surjective. Therefore,  all elements of $\ZZ_p$ of even valuation are norms of the form $aa^\sigma$, and all those with odd valuations are norms of the form $-pbb^\sigma$. This ends the proof that $\Lambda_\ell(1)$ is not bipartite.  

Next we prove the result for general $N$. We show again that $\Lambda_\ell(N)$ has a closed walk of odd length. Let $(E,P)$ be a vertex. Let $\gamma: E \arr E$ be an isogeny of degree $\ell^r$ with $r$ odd as in the previous paragraph. By Lemma \ref{lemma:walkisogeny}, this isogeny provides us with a walk of odd length from $(E,P)$ to $(E,P')$, where $(E,P')\in Rep(N)$ is isomorphic to $(E,\gamma(P)$). In proving the connectivity of the graph, we have shown that there is a walk of even length from $(E,P')$ to $(E,P)$. Composing these walks, we obtain a closed walk of odd length starting and ending at $(E,P)$. This proves that $\Lambda_\ell(N)$ is not bipartite.

We now prove part (3). Suppose that $(N, p)$ is rigid. By counting, it is enough to show that every vertex $(E, P)$ has in-degree at least $\ell+1$. Let $C_1, \dots, C_{\ell+1}$ be the subgroups of order $\ell$ of $E$ and $\pi_i\colon E \arr E/C_i$ the canonical morphism. Consider the diagrams
\[ \xymatrix{(E, \ell^{\alpha(N) - 1} P) \ar[r]^(0.40){\pi_i} & (E/C_i, \ell^{\alpha(N) - 1} P \;{\rm mod }\;C_i) \ar[r]^(0.54){\pi_i^\vee}& (E, \ell^{\alpha(N)} P) = (E, P).}\]
Then, replacing $(E/C_i, \ell^{\alpha(N) - 1} P \;{\rm mod }\;C_i)$ by the representative $(E/C_i, \ell^{\alpha(N) - 1} P \;{\rm mod }\;C_i)^r$ and adjusting $\pi_i^\vee$ by the unique implied $\ell$-isogeny $\pi_i^{\vee,r}$, we find $\ell+1$ arrows leading into $(E, P)$:
\[ \xymatrix{(E/C_i, \ell^{\alpha(N) - 1} P \;{\rm mod }\;C_i)^r \ar[r]^(0.7){\pi_i^{\vee,r}}& (E, P).}\]
We claim that these are distinct arrows. If for some $i, j$ we get the same arrow then we have  $(E/C_i, \ell^{\alpha(N) - 1} P \;{\rm mod }\;C_i)\cong (E/C_j, \ell^{\alpha(N) - 1} P \;{\rm mod }\;C_j)$ \textit{and}, under this (unique) isomorphism, $\pi_i^\vee = \pi_j^\vee$. But then $\pi_i = \pi_j$ and so $C_i = C_j$. 
\end{proof}

\subsection{From walks to endomorphisms}
We introduce some further constructions. In addition to fixing $p, \ell$, $N$, fix now an object $\uE = (E, P)$, where $E$ is a supersingular elliptic curve, and $P$ is a point on $E$ of order  $N$. Let $\calR = \End_{\fpbar}(E)$.

Recall that if $h$ is an arrow from $\uE_1=(E_1,P_1)$ to $\uE_2=(E_2,P_2)$ in $\Lambda_\ell(N)$, then $h$ is an orbit under the action of $\Aut(\uE_2)$ on the set of degree $\ell$ isogenies from $E_1$ to $E_2$ that take $P_1$ to $P_2$. Elements in this orbit were called labels for $h$. By a \textit{labeled walk} in the graph $\Lambda_\ell(N)$ we mean a walk with a choice of label for each arrow in the walk. 
For two closed labeled walks $\omega_1$ and $\omega_2$ based at $\uE$, their composition $\omega_2 \circ \omega_1$ is the labeled walk obtained by traversing $ \omega_1$ followed by~$\omega_2$.  We allow the empty walk and its set of labels is, per definition, $\Aut(\uE)$. We introduce the following notation: 
\vspace{0.4cm}
\begin{itemize}
\item $\Omega(\Lambda_\ell(N), \uE)$ denotes the monoid of closed labeled walks in $\Lambda_\ell(N)$ starting at $\uE$.

\item  Given $\omega\in \Omega(\Lambda_\ell(N), \uE)$, by composing the labels of the arrows occurring in $\omega$, 
we get an endomorphism $\tilde \omega \in \calR$; as $\tilde\omega$ has norm a power of $\ell$,  we have $\tilde\omega\in\scrR^\times$.  

\item We denote by $\scrH_N^+(\uE) $ the monoid in $\scrR^\times$ obtained as the image of $\Omega(\Lambda_\ell(N), \uE)$. By Lemma \ref{lemma:walkisogeny},  $\scrH_N^+(\uE)$ consists of the endomorphisms $f$ of $E$ of degree a non-negative power of $\ell$ such that $f(P) = P$. 
Let $\scrH_N(\uE)$ be the subgroup  of $\scrR^\times$ generated by $\scrH_N^+(\uE)$. Note that $\ell^{\alpha(N)} \in \scrH_N(\uE)$.

\item We denote by $\scrG_N(\uE)$ the group in $\PGL_2(\QQ_q)$ that is the image of $\scrH_N(\uE)$ under the composition $\scrB^\times \arr \GL_2(\QQ_q)  \arr \PGL_2(\QQ_q)$. The map $\scrB^\times \arr \GL_2(\QQ_q)$ is chosen so that the image of $\scrR$ is
\[ \left\{ \left( \begin{smallmatrix} a & pb^\sigma \\ b & a^\sigma
\end{smallmatrix}\right): a, b\in W(\FF_q) 
\right\}, \]
where $\sigma$ is the Frobenius automorphism.
For $\omega\in \Omega(\Lambda_\ell(N), \uE)$, we denote the image of~$\tilde{\omega}$ in $\scrG_N(\uE)$ by $\tilde{\omega}$ again.
\end{itemize}

\vspace{0.4cm}

\begin{lem} $\scrH_N(\uE)$ is the group of rational endomorphisms of $E$ of the form $f = f_1/\ell^t$ such that $f_1\in \End(E), \deg(f_1) = \ell^{r}$, $r, t\in \NN$, and $f(P) = P$.
\end{lem}
\begin{proof} One inclusion is clear: using that for an isogeny $g$, $g^{-1} = g^\vee / \deg(g) $, elements of $\scrH_N(\uE)$   have that property. For the other, note that for $f$ with that property, $g := \ell^{\alpha(N)t} f$ is an endomorphism  satisfying $g(P) = P$ and thus belongs to $\scrH_N^+(\uE)$. The result now follows since $\ell^{\alpha(N)}$ is in $\scrH^+_N(\uE)$. 
\end{proof}

From this point on, we fix $\uE=(E,P)$ in $Rep(N)$ and simplify the notation by dropping $\uE$ from the notation: $\scrH_N=\scrH_N(\uE)$, $\scrG_N=\scrG_N(\uE)$, $\Omega(\Lambda_\ell(N))=\Omega(\Lambda_\ell(N),\uE)$. 

\vspace{0.4cm}
\id It follows easily from the definitions that the  following diagram is commutative and its arrows are homomorphisms that have the indicated properties.  

\[ \xymatrix@!C=1.5cm@M = 0.2cm{
\Omega(\Lambda_\ell(N))\ar[d]^{\pi_N}\ar[r] & \scrH_N \ar@{->>}[r] \ar@{^(->}[d]& \scrG_N \ar@{^(->}[d]	\\
\Omega(\Lambda_\ell(1)) \ar[r] & \scrH_1 \ar@{->>}[r] & \scrG_1	
}
\]

\

If $p \equiv 1 \pmod{12}$, then the pair $(1,p)$ is solid, in the sense of Definition~\ref{dfn: solid rigid}. In particular, $f \mapsto f^\vee$ defines a one-one correspondence between $\Hom(E,E')$ and $\Hom(E',E)$, for any two supersingular elliptic curves $E,E'$. Therefore, we can, just as in the ordinary case, define an undirected graph $\Lambda_\ell^{\rm ud}(1)$ by identifying an arrow corresponding to $f$ with the reverse arrow corresponding to $f^\vee$. Note that when $f$ is an endomorphism of $E$ of degree $\ell$, and $f\neq \pm f^\vee$, the corresponding loop in $\Lambda_\ell^{\rm ud}(1)$ is considered with multiplicity $2$.

Let $H(4\ell)$ denote the Hurwitz class number of positive definite primitive quadratic forms of discriminant $-4\ell$, see \cite{GrossHeights}.
\begin{prop}\label{prop: properties of scrH_n}The following hold: 
\begin{enumerate}
\item The group $\scrH_N$ (resp., $\scrG_N$) is a finitely-generated finite-index subgroup of $\scrH_1$ (resp., $\scrG_1$).
\item The groups $\scrG_N$ are highly non-commutative. To be precise:  \begin{enumerate}
\item The groups $\scrH_N$ and $\scrG_N$ have arbitrarily large subsets of pairwise non-commuting elements. 
\item  Let $p \equiv 1 \pmod{12}$. Let $\gamma(\ell) =  \frac{1}{2}H(4\ell)$. Then the fundamental group $\pi_1(\Lambda_\ell(1), E)$ is a free group of rank  $1+ \frac{\gamma(\ell)}{2} + \frac{(\ell-1) \cdot (p-1)}{24}$. There are elements $\scrC_1, \dots, \scrC_{\gamma(\ell)}$ that are part of a basis of $\pi_1(\Lambda_\ell(1), E)$ such that 
\[ \scrG_1 \cong \pi_1(\Lambda_\ell(1), E)/\langle\!\! \langle \scrC_1^2, \dots, \scrC_{\gamma(\ell)}^2 \rangle\!\!\rangle,\]
where  $\langle\!\!\langle \scrC_1^2, \dots, \scrC_{\gamma(\ell)}^2 \rangle\!\!\rangle$ denotes the minimal normal subgroup containing  $\scrC_1^2, \dots, \scrC_{\gamma(\ell)}^2$.
Consequently, $\scrG_1$ has a quotient that is a free group of rank 
$1 -\frac{\gamma(\ell)}{2} + \frac{(\ell-1) \cdot (p-1)}{24}$.
\end{enumerate}
\end{enumerate}
\end{prop}

\begin{proof} To show that the groups $\scrG_N$ are finitely generated, we will produce a finite set of generators. First, we prove a lemma. 

\begin{lem}\label{lemma:backtrack} Let $\uE_1=(E_1,P_1)$ and $\uE_2=(E_2,P_2)$ be vertices of $\Lambda_\ell(N)$. Let $\lambda\colon\uE_1 \arr \uE_2$ be an isogeny of degree $\ell^m$ for some $m>0$. There is an isogeny $$\eta_\lambda\colon\uE_2 \arr \uE_1$$ of degree $\ell^a\leq \ell^{m+2(\alpha(N)-1)}$ such that  $ \lambda \circ \eta_\lambda$ is multiplication by a power of $\ell$. In particular, for every labeled walk $\omega$ from $\uE_1$ to $\uE_2$ of length $m$, there is a labeled walk $\eta_\omega$ from $\uE_2$ to $\uE_1$ of length $a \leq m+2(\alpha(N)-1)$ such that $\tilde{\omega}\circ\tilde{\eta_\omega}\in \scrG_N$ is the identity.
\end{lem}

\begin{proof} Define $\eta_\lambda:=\ell^{x}\lambda^\vee$, where $x$ is the smallest non-negative integer for which we have  $\alpha(N)|(m+x)$. We have $\deg(\eta_\lambda)=\ell^{m+2x}  \leq  \ell^{m+2(\alpha(N)-1)}$. We also have $\eta_\lambda\circ \lambda=\ell^{x+m}$, and  $\eta_\lambda(P_2)=\ell^{x}(\lambda^\vee(P_2))=\ell^{m+x}P_1 = P_1$, as desired.

\end{proof}

We now prove the finite generation of $\scrG_N=\scrH_G(\uE)$, and deduce the finite generation of $\scrH_N$. 
Let $v(N)$ denote the number of vertices in $\Lambda_\ell(N)$. 
We prove that the finite set 
\[
Gen=\{\tilde{\omega} :   \omega\ {\rm closed\ labeled\ walk\ starting\ at}\ E\ {\rm of\ length\ }\leq 2\alpha(N)+2v(N)-3\}
\]
generates $\scrG_N$.  Let $\omega\in \Omega(\Lambda_\ell(N),\uE)$ be comprised of vertices $E_1=E,E_2,...,E_n,E_{n+1}=E$, and arrows $h_i$ from $E_i$ to $E_{i+1}$, each equipped with a label. For each $2 \leq i\leq n$, let  $\omega_i$ denote a labeled walk from $E_1$ to $E_i$ of length at most $v(N)-1$ (these exist since $\Lambda_\ell(N)$ is connected). Then, $\tilde{\omega}=\tilde{h}_n \cdots \tilde{h}_1$ can be written as a product
\[
\tilde{\omega}=(\widetilde{h_n\omega_n})(\widetilde{\eta_{\omega_n}h_{n-1}\omega_{n-1}}) ...(\widetilde{\eta_{\omega_3}h_2\omega_2})(\widetilde{\eta_{\omega_2}h_1}),
\]
in which each factor is in $Gen$. This proves that $\scrG_N$ is finitely generated. As the kernel of $\scrH_N \arr \scrG_N$ is finitely generated, also $\scrH_N$ is finitely generated. (The kernel is generated by $\ell^{\alpha(N)}$ and $-\ell^{\beta(N)}$, where $\beta(N)$ is the least non-negative integer such that $\ell^{\beta(N)} \equiv -1 \pmod{N}$; if no such $\beta(N)$ exists, the kernel is generated by $\ell^{\alpha(N)}$ alone.)

 Denote by $\scrH_N^1$ the subgroup of $\scrH_1$ composed of elements whose action on the $N$-torsion points of $E$ is trivial. Then, $\scrH_1 \supseteq \scrH_N \supseteq \scrH_N^1$, and to prove $\scrH_N$ is of finite index in $\scrH_1$, it is enough to show that $\scrH_N^1$ is of finite index in $\scrH_1$. But, $\scrH_1/\scrH_N^1\injects \GL_2(\ZZ/N\ZZ)$ and the assertion follows.

\medskip

 Next, we prove that $\scrH_N$ has an arbitrarily large subset of pairwise non-commuting elements. Let $n$ be a positive integer. Let $\calR = \End(E)$, and let $d_1, \dots, d_n$ be large enough distinct primes such that $\ell$ is split and $p$ is inert in each of the fields $K_i = \QQ(\sqrt{-d_i})$. By \cite{EOY} Theorem 1.4 and the discussion preceding it, each field $K_i$ has an optimal embedding $K_i \arr B_{p, \infty}$ relative to $\calR$; that is, the image of $\calO_{K_i}$ is contained in $\calR$. The images of the fields $K_i$ do not commute. 

As $\ell$ is split, in each field $K_i$, we have a decomposition into primes ideals $\ell\calO_{K_i} = \scrL_{i, 1} \scrL_{i, 2}$. Let~$r$ be an integer divisible by the class numbers of $K_1, \dots, K_n$ and let $s_N = \sharp \GL_2(\ZZ/N\ZZ)$. Then, $\scrL_i^r = (t_i)$, and $\tau_i := t_i^{s_N}$  have the following properties: (i) $\tau_i\in \calR$; (ii) $K_i = \QQ(\tau_i)$, and hence $\tau_i$ and $\tau_j$ do not commute for $i \neq j$; (iii) $\tau_i \in \scrH_N$ as $t_i \in \scrH_1$ and $t_i$ acts by an element of $\GL_2(\ZZ/N\ZZ)$ on $E[N]$. This provides us with $n$ pairwise non-commuting elements of $\scrH_N$. In addition, since commutators have trace $0$ and the kernel of $\scrH_N \arr \scrG_N$ consists of non-zero scalars, the images of the $\tau_i$ do not commute in $\scrG_N$ either.

 \medskip 
 Finally, 
 we assume $p \equiv 1 \pmod{12}$, and consider the graph $\Lambda_\ell^{\rm ud}(1)$, which has $(p-1)/12$ vertices. 
The fundamental group $\pi_1(\Lambda_\ell^{\rm ud}(1), E)$ can be described the following way. Let $\Lambda_\ell^{\rm en}(1)$ be the (``enlarged") graph with the same vertices as $\Lambda_\ell^{\rm ud}(1)$, but for every edge $e$ in $\Lambda_\ell^{\rm ud}(1)$ introduce two corresponding edges $e^\epsilon, \epsilon \in \{+, -\}$ in $\Lambda_\ell^{\rm en}(1)$, thereby getting a directed connected graph. Introduce an equivalence relation on the set $\Omega(\Lambda_\ell^{\rm en}(1), E)$ of closed directed walks as the one generated by basic equivalences \[\omega_1\omega_2 \sim \omega_1e^\epsilon e^{-\epsilon} \omega_2.\] The set of equivalence classes forms a group, naturally isomorphic to $\pi_1(\Lambda_\ell^{\rm ud}(1), E)$. Note that we have graph morphisms 
\[\Lambda_\ell^{\rm en}(1) \arr \Lambda_\ell(1) \arr \Lambda_\ell^{\rm ud}(1).\] 
In fact, the morphism $\Lambda_\ell^{\rm en}(1) \arr \Lambda_\ell(1)$ is a bijection on vertices and all edges except for \textit{self-dual edges} -- namely, edges that correspond to loops with labels $\pm \lambda$ such that $\pm \lambda = \pm \lambda^\vee$ (i.e., the loops of multiplicity~$1$ in $\Lambda_\ell^{\rm ud}(1)$). There is a natural map 
\[ \Omega(\Lambda_\ell^{\rm en}(1), E) \arr \scrG_1.\]
Indeed, each edge $e^\epsilon$ inherits a set of labels $\pm \lambda$ from $\Lambda_\ell(1)$, and so for each walk we can associate the composition of the labels that gives a well defined element in $\scrG_1$. Equivalent elements of $ \Omega(\Lambda_\ell^{\rm en}(1), E)$ have the same image. We note that if $e^\epsilon$ corresponds to a self-dual loop in $\Lambda_\ell(1)$ then we get $(e^\epsilon)^2=1$ in $\scrG_1$. 

Define for every loop $c_i$ of multiplicity~$1$ in $\Lambda_\ell(1)$ an element of $\pi_1(\Lambda_\ell^{\rm ud}(1), E)$ by taking any path $\omega_i$ from $E$ to the vertex of the loop $c_i$, and letting $\scrC_i = \omega_i^\vee c_i \omega_i$. We get a well-defined homomorphism
\[\pi_1(\Lambda_\ell(1), E)/\langle\!\!\langle \scrC_1^2, \dots, \scrC_{\gamma(\ell)}^2 \rangle\!\!\rangle \arr \scrG_1.\]
It is easy to check using Lemma~\ref{lemma:walkisogeny} that this is a surjective homomorphism. The injectivity amounts to checking that a composition of $\ell$-isogenies $f_n \circ \dots \circ f_1$ is, up to a sign, multiplication by a power of $\ell$ if and only if it is composed of nested pairs of $\ell$-isogenies $\pm\lambda\circ \lambda^\vee$. To see this last statement, consider a composition of $\ell$-isogenies $f_n \circ \dots \circ f_1$ which is equal to $\ell^{n/2}$, say 
\[ \xymatrix{E = E_0 \ar[r]^(.62){f_1} & E_1 \ar[r]^{f_2}  & \dots  \ar[r]^{f_n}& E_n}.\]
Let $i+1$ be the minimal integer such that $\ell$ divides $f_{i+1} \circ \dots \circ f_1$. That is, $\Ker(f_{i+1} \circ \dots \circ f_1) \cong \ZZ/\ell^i \ZZ \times \ZZ/\ell \ZZ$ and, under this isomorphism, $\Ker(f_{j} \circ \dots \circ f_1) = \ZZ/\ell^j \ZZ \times \{ 0 \}$ for $j \leq i$. It follows that $\Ker(f_{i+1} \circ f_i) \cong (\ZZ/\ell \ZZ)^2$, which implies that $E_{i+1} \cong E_{i-1}$ and $f_{i+1} \circ f_i$ is equal to the multiplication-by-$\ell$ map up to an automorphism of $E_{i+1}$, namely, up to $\pm 1$ due to our assumption on $p$. Therefore, $f_{i+1} = \pm f_i^\vee$.

A label of a self-dual loop at a vertex $E^\prime$ corresponds to an embedding $\ZZ[\sqrt{-\ell}]\injects \End(E^\prime)$. By \cite[\S1]{GrossHeights}, there are $H(4\ell)$ such embeddings. Therefore 
the number of self-dual loops of $\Lambda_\ell^{\rm ud}(1)$ is $\gamma(\ell)$.
Hence, the group $\pi_1(\Lambda_\ell^{\rm ud}(1), E)$ is a free group of rank $1+ \frac{\gamma(\ell)}{2} + \frac{(\ell-1) \cdot (p-1)}{24}$ and so $\scrG_1$ has a quotient that is a free group of rank $1- \frac{\gamma(\ell)}{2} + \frac{(\ell-1) \cdot (p-1)}{24}$.
\end{proof}

\subsection{The closure of $\scrH_N^+$, $\scrH_N$ and $\scrG_N$} Our purpose is to determine the $p$-adic and Zariski closures of the monoid $\scrH_N^+$ in $\scrR^\times$, and the same for $\scrG_N$. We will use $\HH, \GG$, etc. to denote $p$-adic closures and $H, G$, etc. to denote Zariski closures.

We introduce the following submonoid of $\scrH_N^+$:
\[ \scrH_N^{1, +} = \{ f \in \End(E): f \equiv 1 \!\!\!\pmod{N}, \deg(f)\in \langle \ell^{\alpha(N)} \rangle \},\]
and let $\scrH_N^1$ be the minimal group containing $\scrH_N^{1, +}$ (it already appeared in the proof of Proposition~\ref{prop: properties of scrH_n}).
Denote by $\HH_N^+, \HH_N^{1, +}$ the $p$-adic closure of $\scrH_N^+$ and $\scrH_N^{1,+}$ in $\scrR^\times$, respectively.
Let $\scrL_N$  be the $p$-adic closure of $\langle \ell^{\alpha(N)}\rangle$ in $\scrR^\times$. Note that $\scrL_N$ is also the $p$-adic closure of $\{ \ell^{r\alpha(N)}: r = 1, 2, 3, \dots\}$
\begin{prop} $\HH_N^{1, +} = \{ x\in \scrR^\times: {\rm Nm}(x) \in \scrL_N\}$.
\end{prop}
\begin{proof} Let $A =  \{ x\in \scrR^\times: {\rm Nm}(x) \in \scrL_N\}$. 
Clearly, $\scrH_N^{1, +} \subset A$. Since $A$ is a compact group, also $\HH_N^{1, +}\subseteq A$.

For the reverse inclusion we use Theorem~1.2 of \cite{Sar}. Write $\calR = \End(E)\subset B_{p, \infty}$ and identify $\calR$ with $\ZZ^4$ so that $1 = (1, 0, 0, 0)$; let $Q$ be the positive definite quadratic form corresponding to the norm of $B_{p, \infty}$. Then $\scrR = \ZZ_p^4$. Let $\underline{\alpha} = (\alpha_1, \dots, \alpha_4) \in A$. We need to show that for all $M \geq 0$, $\exists (\beta_1, \dots, \beta_4) \in \ZZ^4$ such that:
\begin{itemize}
\item $Q(\beta_1, \dots, \beta_4) = \ell^{r\alpha(N)}$ for $r>0$ such that  $\ell^{r\alpha(N)} \equiv {\rm Nm}(\underline{\alpha}) \pmod{p^{M}}$, 
\item $(\beta_1, \dots, \beta_4) \equiv (1, 0, 0, 0) \pmod{N}$, and
\item $(\beta_1, \dots, \beta_4) \equiv  (\alpha_1, \dots, \alpha_4) \pmod{p^M}$.
\end{itemize}
Note that there are always arbitrary large $r$ satisfying the congruence in the first condition. For convenience we take $M \geq 2$.

Let $m = Np^M$ and let $(\lambda_1, \dots, \lambda_4) \in (\ZZ/m\ZZ)^4$ correspond to the above last two congruence conditions. 
By the result cited above, to get a solution (for $r\gg_M 0$), it suffices to show that for every prime $u$ there is a  $(x_1, \dots, x_4) \in \ZZ_u^4$ such that 
$Q(x_1, \dots, x_4) = \ell^{r\alpha(N)}$ and $(x_1, \dots, x_4) \equiv (\lambda_1, \dots, \lambda_4)\pmod{u^{\ord_u(m)}}$.

For $u\nmid m$, there is no congruence condition, $(\ZZ_u^4, Q) \cong (M_2(\ZZ_u), \det)$, and the matrix $\left(\begin{smallmatrix}1 & \\ & \ell^{r\alpha(N)}
\end{smallmatrix}\right)$ is a solution. For $u \vert N$, we still have $(\ZZ_u^4, Q) \cong (M_2(\ZZ_u), \det)$, and the matrix $\left(\begin{smallmatrix}1 & \\ & \ell^{r\alpha(N)}
\end{smallmatrix}\right)$ is still a solution because it satisfies the congruence condition $\left(\begin{smallmatrix}1 & \\ & \ell^{r\alpha(N)}
\end{smallmatrix}\right) \equiv \left(\begin{smallmatrix}1 & \\ & 1
\end{smallmatrix}\right) $ mod ${u^{\ord_u(m)}}$.

For $u = p$ we want a solution in $\ZZ_p^4$, congruent to $(\alpha_1, \dots, \alpha_4) \pmod{p^M}$. We have now $(\scrR, Q) \cong (\{\left( \begin{smallmatrix} a& pb^\sigma \\ b & a^\sigma \end{smallmatrix}\right): a, b \in W(\FF_q)\}
, \det)$. The element $\underline{\alpha}$ obviously satisfies the congruence condition, is represented by a matrix $\left( \begin{smallmatrix} a_1& pb_1^\sigma \\ b_1 & a_1^\sigma \end{smallmatrix}\right)$,  and has determinant ${\rm Nm} (\underline{\alpha})$ that is congruent to $\ell^{r\alpha(N)}$ modulo $p^{M}$. We want to find $\tilde a_1$ such that $\tilde a_1\equiv a_1 \pmod{p^M}$ and $\det\left( \begin{smallmatrix} \tilde a_1& pb_1^\sigma \\ b_1 & \tilde a_1^\sigma \end{smallmatrix}\right) = \ell^{r\alpha(N)}$.

We claim that the norm map
\[ a_1 + p^MW(\FF_q) \arr {\rm Nm}(a_1) + p^M\ZZ_p\]
is surjective. As $a_1$ is a unit, $a_1 + p^MW(\FF_q) = a_1(1 + p^MW(\FF_q))$ and it is enough to prove that ${\rm Nm}$ is a surjection $1 + p^MW(\FF_q) \arr 1 + p^M\ZZ_p$.

Applying the logarithm to both sides we are reduced to showing $tr\colon  p^{M+1}W(\FF_q) \arr p^{M+1}\ZZ_p$ is surjective, where $tr$ denotes the trace map $tr_{W(\FF_q)/\ZZ_p}$; this is easily verified. Finally, since $\ell^{r\alpha(N)} + pb_1b_1^\sigma \in {\rm Nm}(a_1) + p^M\ZZ_p$, there is an $\tilde a_1\in a_1 + p^MW(\FF_q)$ whose norm is $\ell^{r\alpha(N)} + pb_1b_1^\sigma$.
\end{proof}

\begin{cor} $\HH_N^{1, +}$ is a group, equal to the $p$-adic closure of $\scrH_N^1$. Denote it henceforth $\HH_N^1$. The quotient group $\scrR^\times/\HH_N^1$ is a finite group isomorphic to $\ZZ_p^\times/\scrL_N$.
\end{cor}
We have inclusions $\HH_1 = \HH_1^1 \supseteq \HH_N^+ \supseteq \HH_N^1$, and so $\HH_N^+ / \HH_N^1$ is a finite monoid contained in the group $\HH_1/\HH_N^1$, hence a group itself. 
\begin{cor}$\HH_N^+$ is a group, equal to the $p$-adic closure of $\scrH_N$. Denote it henceforth $\HH_N$. We have $\HH_N/\HH_N^1$ is a cyclic subgroup of $\scrL_1/\scrL_N \cong \ZZ/\alpha(N)\ZZ$.
\end{cor}
We are interested in the image of $\HH_N$ under the projection map $\Pi\colon \GL_2(\QQ_q) \arr \PGL_2(\QQ_q).$
Denote \[sat(\HH_N) = \ZZ_p^\times \HH_N = \Pi^{-1}\Pi(\HH_N) \cap \scrR^\times.\] 
The group $sat(\HH_N)$ is closed, hence $\Pi(\HH_N)$ is closed in $\PGL_2(\QQ_q)$ and equal to $\GG_N$.

\begin{prop} We have
\[ sat(\HH_N) = \scrR^\times_\square:=\left\{ \left(\begin{smallmatrix} a & pb^\sigma\\ b & a^\sigma
\end{smallmatrix}\right)\in \scrR^\times \colon  aa^\sigma - pbb^\sigma \in (\ZZ_p^\times)^2 \cdot \langle \ell \rangle\right\}.\]
Thus, $sat(\HH_N)$ has index at most $2$ in $\scrR^\times$ if $p$ is odd, and index at most $4$ if $p=2$; further,  $sat(\HH_N) = \scrR^\times$ if and only if $p$ is odd and $\ell$ is not a square modulo $p$.
\end{prop}
\begin{proof} We first examine $sat(\HH_N^1)$. As $\scrR^\times /\HH_N^1 \cong \ZZ_p^\times/\scrL_N$ via the determinant map, we find that  $\scrR^\times /sat(\HH_N^1)$ is isomorphic to $\ZZ_p^\times/(\ZZ_p^\times)^2\langle \ell^{\alpha(N)} \rangle$. Since the graph $\Lambda_\ell(N)$ is not bipartite (Theorem~\ref{thm: properties of the graphs}), there is an element in $\scrH_N^+$ whose degree is an odd power of $\ell$. Consequently, $\scrR^\times/sat(\HH_N)\cong \ZZ_p^\times/(\ZZ_p^\times)^2\langle \ell \rangle$ and, in fact, $sat(\HH_N)$ is precisely the elements in $\scrR^\times$ whose determinant lies in $(\ZZ_p^\times)^2\langle \ell \rangle$.
\end{proof}

\begin{cor} \label{cor: description of G_N} We have $\GG_N = \Pi(\scrR^{\times}_{\square})$, which is independent of $N$.
\end{cor}

\begin{prop}\label{prop: Zariski closure} Let $H_N$ and $G_N$ be the Zariski closures of $\scrH_N, \scrG_N$ in $\GL_2(\QQ_q)$ and $\PGL_2(\QQ_q)$, respectively.  Then $H_N = \GL_2(\QQ_{q})$ and $G_N = \PGL_2(\QQ_{q})$. 
\end{prop}

\begin{proof}
As $H_N \supseteq \HH_N, G_N \supseteq \GG_N$, our assertion is quite clear from the description of algebraic subgroups of $\GL_2$ and $\PGL_2$ over an algebraic closure. For example, the proof of Proposition~\ref{prop: properties of scrH_n}  shows that we can find arbitrarily large subsets of pairwise non-commuting elements of $\scrH_N$ lying in the connected component $H_N^0$ of $H_N$. Indeed, if $h= [H_N:H_N^0]$, we may replace in the proof the elements $\tau_i$ by $\tau_i^h$ that are mutually non-commuting, semi-simple, and have semi-simple commutators (every element of $\scrB^\times$ is semi-simple). It is also easy to arrange for a commutator to be non-central. From this it is easy to deduce that $H_N^0 = \GL_2$ or $\SL_2$. On the other hand, we have elements of determinant $\ell^{\alpha(N)m}, m \in \ZZ,$ in $H_N$. It follows that $H_N = \GL_2$ and $G_N = \PGL_2(\QQ_q)$.
\end{proof}

\subsection{Dynamics implications - I}  At this point we want to explain the implication of our results thus far for the study of the dynamics of the Hecke operator $T_\ell$. Let $\uE = (E, P)$ be an elliptic curve over $\CC_p$ with good supersingular reduction, thought of as a point on $X_1(N)(\CC_p)$ for $N$ rigid. Let $D(\uE)$ be its residue disc.  As we have seen (the arguments proving Lemma~\ref{lem: branches} apply), the correspondence $T_\ell^n\vert_{D(\uE)}$ decomposes into a sum of rigid-analytic isomorphisms between residue discs. If $T_\ell^n(\uE) = \sum_i [\uE_i]$ (with repetitions allowed) then $T_\ell^n\vert_{D(\uE)} = \sum_i f_i$, where each $f_i\colon D(\uE) \arr D(\uE_i)$ is an isometry. Further, the discs $D(\uE_i)$ are recognized by the mod $p$ reductions $\bar \uE_i$, which are supersingular curves too. Thus, as in the ordinary case, the movement of the residue discs follows the dynamics of the Hecke orbit of $\bar \uE$, which is described by walks in the finite supersingular graph $\Lambda_\ell(N)$.

Enumerate the vertices of the $(\ell+1)$-regular directed graph $\Lambda_\ell(N)$ as $\{v_1, \dots, v_{\nu(N)}\}$, and let~$T$ be the normalized adjacency matrix, so that $(\ell+1)T_{ij}$ is the number of edges from $v_i$ to~$v_j$. By Theorem~\ref{thm: properties of the graphs}, the matrix $T$ is bi-stochastic and can be viewed as a time-homogenous irreducible Markov chain. As $\Lambda_\ell(N)$ is not bipartite, every eigenvalue of $T$, besides $1$, has modulus smaller than $1$. Thus, any initial probability distribution on $\Lambda_\ell(N)^V$ converges exponentially fast to the unique probability distribution $\frac{1}{\nu(N)}\sum_i \delta_{v_i}$, where $\delta_{v_i}$ is the Dirac distribution supported on $v_i$. (For $N=2$ everything goes through the same; for $N=1$ most of the considerations go through, but the limit measure need not be the one we have specified. For example, if $p=11$ and $\ell = 5$, then $T$ is $\frac{1}{6}\cdot\left( \begin{smallmatrix} 3 & 3 \\ 2 & 4 \end{smallmatrix} \right)$ and the stationary distribution is $\left( \begin{smallmatrix} 2/5 \\ 3/5 \end{smallmatrix} \right)$.)

With these considerations, to understand further the action of $T_\ell^n$ we may restrict our attention to a single disc.
That is, we need to study the action of the monoid $\scrH_N^+(\uE)$ on the residue disc $D(\uE)$. This residue disc may be viewed as the generic rigid-analytic fibre $(\tilde D(\bar \uE))_{rig}$ of the formal scheme $\tilde D(\bar \uE)$ that is the formal deformation space of the elliptic curve $\bar \uE$; equivalently, by Serre-Tate and \'etaleness, of the $p$-divisible group $\bar{E}[p^\infty]$; equivalently, by Tate, of the formal group~$\hat{\bar E}$.
The group $\scrR^\times$, and hence $\scrH_N$, can thus be interpreted as automorphism groups of~$\hat {\bar E}$. The work of Gross-Hopkins, which is our next topic, will allow us to consider instead the action of $\scrH_N$, that factors through $\scrG_N$, on $\PP^1_{\QQ_q}$. The methods we shall use will rely on the $p$-adic and Zariski closures of $\scrG_N$. As the groups $\GG_N, G_N$ are well-understood by Corollary~\ref{cor: description of G_N} and Proposition~\ref{prop: Zariski closure}, this will prove to be a powerful technique.

 We can view the formal deformation space $\tilde D(\bar \uE)$ as representing the functor $\scrF$ on complete local Noetherian $W(\FF_{q})$-algebras $R$ with residue field $k$,
\begin{equation}\label{eqn: scrF}
\scrF:\quad  R \mapsto \scrF(R) = \{ (\scrE, \varphi) \}/isom.,
\end{equation}
where $\scrE$ is a $p$-divisible group over $\Spec(R)$, and $\varphi\colon  \scrE \otimes_R k \arr E[p^\infty]\otimes_{\FF_{q}} k$ is an isomorphism. The group $\Aut(E[p^\infty])$, of which $\scrH_N$ is a subgroup, acts on the universal deformation space via its action on the functor:  $\gamma \in \Aut(E[p^\infty])$ acts by \[ (\scrE, \varphi) \mapsto (\scrE, \gamma\circ \varphi).\]
Since the $p$-divisible group of $\bar E$ is connected, by a result of Tate, the functor of deformations of the formal group of $\bar E$ is equivalent to the function $\scrF$.

\medskip

\id For $N=1$ one uses the following diagram. Choose an auxiliary rigid $N$, and use the diagram:
\[\xymatrix@!R=0.5cm{D(\uE)\ar[d] & = &(\tilde{D} (\bar \uE))_{rig}\ar[d]^\cong_{\scrH_N^+-equivar.}\\
D(E)\ar[d] & =& (\tilde{D}(\bar E))_{rig}\ar[d]\\
D(E) /\Aut(\bar E)& = &(\tilde{D}(\bar E))_{rig}/\Aut(\bar E)
}\]
 We remark that the action of $\Aut(\bar E)$ is non-trivial only when $\Aut(\bar E) \supsetneqq \{ \pm 1\}$, and this only happens at $j = 0, 1728$, if these are supersingular at all. 

\subsection{The work of Gross and Hopkins} In this section we review some of the work of Gross and Hopkins \cite{GH}, specializing to cases of particular interest to us. 

\subsubsection{The ideal disc}\label{sec: the ideal disc} Gross and Hopkins study one-dimensional connected $p$-divisible groups through one dimensional formal groups. In \cite[p. 31]{GH} they construct a very particular universal typical formal group $F$ of height $n$ over the ring $\ZZ_p[\![u_1, \dots, u_{n-1}]\!]$. The case that would interest us most is $n=2$. Let $\bar F_0$ denote the reduction of $F$ modulo the maximal ideal of $\ZZ_p[\![u_1, \dots, u_{n-1}]\!]$; it is a formal group of dimension $1$ and height $n$ over $\FF_p$, and, as in \cite{GH}, we let $F_0$ be the specialization of $F$ to a formal group over $\ZZ_p$ obtained by specializing all the $u_i$ to $0$. We denote $\bar F = \bar F_0\otimes_{\FF_p} \FF_{p^n} $. (Note that our notation differs a little from \cite{GH}.)

Consider the functor $\scrF^\star$ on the category of complete local noetherian $W(\FF_{p^n})$-algebras $(R, \germ_R)$ with $p\in \germ_R$, $k_R \colon =R/\germ_R$, given by
\begin{equation}\label{eqn: scrFstar}
\scrF^\star: \quad  R \mapsto \scrF^\star(R)=\{ G/R \text{ a formal group law\;} G\otimes_R k_R = \bar F \otimes_{\FF_{p^n}}k_R \} / \star-{\text isom.}
\end{equation}
Here, we say that $G/R$ and $G'/R$ are $\star$-{\textit{isomorphic}} if there is an isomorphism $G \overset{\sim}{\arr} G'$ over $R$, which reduces to the identity modulo $\germ_R$. This functor is representable by $W(\FF_{p^n})[\![u_1, \dots, u_{n-1}]\!]$, and~$F\otimes_{\ZZ_p}W(\FF_{p^n})$ is a universal object over it \cite[p. 45]{GH}. 
 
 Let $A_n = W(\FF_{p^n})$; it is an unramified extension of $\QQ_p$ of degree $n$ that contains all $p^n -1$ roots of unity. From the specific definition of $F$ given in \cite{GH}, it follows that $F_0 \otimes_{\ZZ_p} A_n$  has an action of $A_n$ such that $[\zeta](x) = \zeta \cdot x$ for $\zeta\in\mu_{p^n-1}$.
This induces an action of $A_n$ on $\Fbar$, making it into a formal $A_n$-module of height $1$, with an induced embedding $j_n\colon A_n \arr \End(\Fbar)$.
The image of $j_n$ along with the Frobenius morphism $\varphi$ generate $\End(\Fbar)$. Indeed, we have  
\[
\End(\Fbar \otimes_{\FF_{p^n}} \fpbar)= \End(\Fbar) = A_n \oplus A_n\varphi \oplus \cdots \oplus A_n\varphi^{n-1},
\]
with $\varphi^n = p$, and $\varphi a = a^\sigma \varphi$, for $a \in A_n$, where $\sigma$ denotes the Frobenius automorphism of~$A_n$. The centre of $\End(\Fbar)$ is $\ZZ_p$. Note that every element $b$ of $\End(\Fbar)$, being an endomorphism of a $1$-dimensional formal group, is given by a power series $b(x) \in \FF_{p^n}[\![x]\!]$.

\subsubsection{The action of $\Aut(\bar F)$ on $\Spf A_n[\![u_1, \dots, u_{n-1}]\!]$} From the above discussion it follows that  
\[ \scrA_n: = \Aut(\bar F) = A_n^\times \oplus A_n\varphi \oplus \cdots \oplus A_n\varphi^{n-1},\]
and its center is $\ZZ_p^\times$. 
The action of $\scrA_n$ on $\Spf A_n[\![u_1, \dots, u_{n-1}]\!]$ is already described in \cite{LT} as follows:  Given $b\in \scrA_n$, represented by the power series $b(x) \in \FF_{p^n}[\![x]\!]$,
let $b^{-1}(x)$ be the inverse of $b(x)$ relative to composition (so that $b(b^{-1}(x)) = x$, etc.) and lift $b^{-1}(x)$ to a power series $h(x) \in A_n[\![ x ]\!]$ in an arbitrary way. Given the universal formal group law $F$ over $A_n[\![ u_1, \dots, u_{n-1}]\!]$, there is a unique group law $F^\prime$ over $A_n[\![ u_1, \dots, u_{n-1}]\!]$ such that
\[ F^\prime(h(x), h(y)) = h(F(x, y)).\]
Note that $F\overset{\sim}{\Arr} F^\prime$ over the ring $A_n[\![ u_1, \dots, u_{n-1}]\!]$ via the map $h$, but this isomorphism is not a $\star$-isomorphism. By the universal property of $F$, there exists a unique automorphism $\beta(b)$ of $A_n[\![ u_1, \dots, u_{n-1}]\!]$ over $A_n$  such that 
\[ \beta(b) F \underset{\star}{\simeq} F^\prime. \]
(The notation $\beta(b)F$ means specializing $F$ by applying the $\beta(b)$ to its coefficients.) 
The automorphism $\beta(b)$ depends only on $b$ and not on the choice of $h$, since for varying $h$  the formal group laws $F'$ are all $\star$-isomorphic.

Here we have described the global action, but the action on points on the universal deformation space $A_n[\![ u_1, \dots, u_{n-1}]\!]$, namely on specializations of $F$ through homomorphisms $A_n[\![ u_1, \dots, u_{n-1}]\!]\arr R$, is virtually the same. 

This describes the action of $\scrA_n$ via the $\star$-isomorphism model. If we consider the equivalent functor classifying
$(G/R, \psi)$ a formal group law with $\psi\colon G\otimes_R k_R \overset{\cong}{\Arr} \bar F \otimes_{\FF_p}k_R$, up to isomorphism, we have a natural action of $\scrA_n$, where $b\in \scrA_n$ acts on $(G/R, \psi)$ by taking it to $(G/R, b\circ \psi)$. The following proposition is an easy exercise.

\begin{prop}\label{prop: equivalence of alpha actions}  The two actions of $\scrA_n$ on the universal deformation space $A_n[\![u_1, \dots, u_{n-1}]\!]$ are the same. 
\end{prop}

We thus get a homomorphism $\beta$, with kernel $\ZZ_p^\times$,
\[ \beta\colon  \scrA_n \Arr \Aut_{A_n}(A_n[\![u_1, \dots, u_{n-1}]\!]).\]
It induces an action of $\scrA_n$ on the formal scheme $\tilde{X}=\Spf(A_n[\![u_1, \dots, u_{n-1}]\!])$ and on the associated rigid space 
\[ X = \tilde{X}_{rig}= \left(\Spf(A_n[\![u_1, \dots, u_{n-1}]\!])\right)_{rig}.\]
In \cite[Proposition 14.13]{GH} we find that if $R$ is a complete local noetherian $A_n$-algebra, and $x\in \tilde{X}(R)$ with corresponding formal group law $F_x$, then the orbit of $x$ under $\scrA_n$ is 
\[ {\rm Orb}(x) = \{ y\in \tilde{X}(R)\colon  F_x \cong F_y \text{ over } R\}\]
(note that this is, of course, an isomorphism and not a $\star$-isomorphism) and 
\[ \Stab(x) = \Aut(F_x).\] 

\subsection{Canonical and quasi-canonical lifts: the supersingular case} 
We consider now the case $n=2$. Our main interest in this section is in points $x\in X$ that have a large endomorphism ring. They are tightly linked to dynamical aspects of the action of $\Aut(\Fbar)$ on $X$, and a fortiori to the action of $\scrH_N$ on $X$, as we shall see. But we initially consider the larger picture, which is the action of $\Aut(\Fbar)$. We consider points $x\in X$ such that
\[ \End(F_x) \supsetneqq \ZZ_p.\]
In this case, $\End(F_x)$ is an order in a quadratic field extension of $\QQ_p$. Up to isomorphism there are finitely many such fields; they are in bijection with the non-identity elements of $\QQ_p^\times/(\QQ_p^\times)^2 \cong \ZZ/2\ZZ \times \ZZ_p^\times/(\ZZ_p^\times)^2$. 
Fix models $\{ K_i\} = \{ K_1, \dots, K_r\}$
for these fields, and for $K \in \{ K_i\}$ denote the ring of integers by $\calO_{K}$. Let $\{ 1, \gamma_{K}\}$ be the automorphism group of $K/\QQ_p$.

\subsubsection{} We first classify $x\in X$ according to the isomorphism class $K_i$ of the field $K_x:= \End^0(F_x)$.  

\vspace{0.2cm}

\id Fix then $K \in \{ K_i\}$ and consider points $x$ such that $K_x \cong K$.
We have then
\[ \End(F_x) \cong \calO_{K, s_x}:=\ZZ_p + p^{s_x} \calO_K, \]
and we call the non-negative integer $s_x$ the \textit{level}, or \textit{conductor}, of $x$. Choose an isomorphism $\iota_x\colon \calO_{K, s_x} \arr \End(F_x)$. It is unique up to composition with $\gamma_K$; when we want to refer to the pair of embeddings we shall write $\iota_x^\pm$. The map $\iota_x$ induces  an embedding 
\[ \bar \iota_x\colon \calO_{K, s_x} \hookrightarrow \End(\bar F) \cong \scrR.\]
As the order $\scrR$ consists of all integral elements of $\scrB$, the embedding $\bar \iota_x$ extends uniquely to an embedding $\bar \iota\colon \calO_K \hookrightarrow \End(\bar F)$. Thus, any embedding $\bar \iota_x$ comes from an ``optimal embedding".

Given $d\in \scrR^\times \setminus \ZZ_p^\times$, it will be convenient to denote $\calO_{K, s(d)}$ the order
isomorphic to $\ZZ_p[d]$, and denote $\iota_d$ its embedding into $\End(\bar F)$.

\subsubsection{} Fix an embedding $\bar \iota\colon \calO_K \arr \End(\bar F)$. Let $ QC(\bar \iota, K, s)$ be the set of points $x$ of $X$ such that $K_x \cong K$, $s_x = s$ and $\iota_x$ extends to $\bar \iota$. These are the quasi-canonical lifts relative to $\bar \iota$ of level $s$. Let $qc(K,s) = \sharp  QC(\bar \iota, K, s)$ (it is independent of $\bar \iota$). Gross \cite{qcan} gives the following formula:
\begin{equation}\label{eqn: qc(K,s)}
qc(K,s) = \begin{cases} 
 1 & {\rm if}\; s=0, \\ (p+1)p^{s-1} & s>0, K \;\text{\rm  unramified}, \\  p^s & s>0, K \;\text{\rm ramified}.
 \end{cases}
\end{equation}

\medskip

\id \textit{Caution:} Note that what Gross calls a quasi-canonical lift $x$ of level $s$ is a point of $QC(\bar \iota, K, s)$ if $s$ is even, and a point of  $QC(\varphi \bar \iota \varphi^{-1}, K, s)$ if $s$ is odd, where $\varphi\colon \bar F \arr \bar F$ is  Frobenius. 
Nonetheless, the result about $qc(K, s)$ quoted from \cite{qcan} is valid, as $\varphi \bar \iota \varphi^{-1}$ is an embedding $\calO_K \arr \scrR$ as well. 

\subsubsection{} The sets of points we have defined enjoy an action of $\scrB^\times$ coming from conjugating the embedding $\bar \iota$. 
An element $\gamma \in \scrB^\times$ induces, for any $K$ and $s$, a bijection
\[ QC(\bar \iota, K, s) \overset{\sim}{\Arr} QC(\gamma \bar \iota\gamma^{-1}, K, s).\]
We are interested in the action of $\scrR^\times= \Aut(\Fbar)$. Consider the surjective homomorphism $\nu:= \ord \circ {\rm Nm}: \scrB^\times \arr \ZZ$,
where ${\rm Nm}$ is the reduced norm and $\ord$ is $p$-adic valuation. In the model $\left\{ \left( \begin{smallmatrix} a & pb^\sigma \\ b & a ^ \sigma
\end{smallmatrix}\right) \right\}$, $\nu= \ord \circ \det$ and so we see that $\scrR^\times = \Ker(\nu)$ and $\scrB^\times /\scrR^\times \cong \ZZ$. As by Skolem-Noether $\scrB^\times$ acts transitively on the set of embeddings $\bar \iota\colon K \arr \scrB$ (they are all optimal), the orbits of $\scrR^\times$ on the set of these embeddings are in bijection with $\scrB^\times/\scrR^\times\cdot {\rm Cent}_{\scrB^\times}(\bar \iota(K)) = \scrB^\times/\scrR^\times\cdot \bar \iota(K^\times)$. As $\nu\vert_K = \ord \circ {\rm Nm}_{K/\QQ_p}$ we find that the $\scrR^\times$-orbits in the set of optimal embeddings $\calO_K \arr \scrR$ are in bijection with $\ZZ/\ord ({\rm Nm}({K^\times}))$. Therefore:
\begin{itemize}
\item If $K/\QQ_p$ is ramified, $\Aut(\Fbar)$ acts transitively on the set of embeddings $\bar \iota\colon \calO_K \arr \scrR$. In this action, the centralizer ${\rm Cent}(\bar \iota)$ is $\bar \iota (K^\times)$ and the normalizer satisfies $ {\rm Norm}(\bar \iota)/ {\rm Cent}(\bar \iota) \cong \{ 1, \gamma_K\}$, where the non-trivial class indeed induces $\gamma_K$ by its conjugation action. We remark that this persists also when we take the normalizer and centralizer in $\scrR^\times$. 
\item If $K/\QQ_p$ is unramified, there are two orbits for the action of $\scrR^\times$ on the set of embeddings $\bar \iota\colon \calO_K \arr \scrR$. Again, $ {\rm Norm}(\bar \iota)/ {\rm Cent}(\bar \iota) \cong \{ 1, \gamma_K\}$, but if we take the normalizer and centralizer in $\scrR^\times$ we have $ {\rm Norm}(\bar \iota)\cap \scrR^\times/ {\rm Cent}(\bar \iota)\cap \scrR^\times \cong \{ 1\}$. To see that, consider the standard inclusion $K \arr \scrB$ given by $a\mapsto  \left( \begin{smallmatrix} a & 0\\ 0 & a^\sigma
\end{smallmatrix}\right)$, where the normalizer is generated over the centralizer by $\varphi = \left( \begin{smallmatrix} 0 & p\\ 1 & 0
\end{smallmatrix}\right)$ -- a non-integral element. The two $\scrR^\times$ orbits are thus distinguished by the induced action on the tangent space of $\Fbar$. Note, however, that even in this case, $\Aut(\Fbar)$ acts transitively on the subrings of $\scrB$ that are isomorphic to $\calO_K$.
\end{itemize}
 
\subsubsection{}\label{subsec: 1} Consider now $b\in \scrR^\times$ and the automorphism $\beta(b)$ of $X$ induced by it. If $b\in \ZZ_p^\times$ then $\beta(b)$ is the identity map, and conversely. We are interested in the periodic points $\text{Per}(b)$  of $\beta(b)$ (in particular in the fixed points $\text{Fix}(b)$), and so we may assume $b\not\in \ZZ_p^\times$. 
\begin{itemize}
\item If $b$ has finite order in $\scrR^\times/\ZZ_p^\times$, then $\text{Per}(b) = X$.
\item
If $b$ has infinite order in $\scrR^\times/\ZZ_p^\times$ and $x\in X$ is a periodic point of length $r$, say, then $x$ is a fixed point of $\beta(b^r)$. Note that there are  two embeddings associated to $b^r$, 
\[ \bar \iota_{b^r}^\pm\colon \calO_{K, s} \arr \scrB, \]
where $K\in \{ K_i\}$ is the unique field isomorphic to $\QQ_p(b^r) = \QQ_p(b)$ and  $s = s(b^r)$ is the level. We therefore  next direct our attention to fixed points. 
\end{itemize}
\subsubsection{}\label{subsec: 2} Let $d\in \scrR^\times, d\not\in \ZZ_p^\times$. As noted  $\text{Fix}(\beta(d)) = \{ x : d \in \End(F_x)\}$.
If so, $\End(F_x)$ is an order of conductor $s_x \leq s(d)$. Furthermore, the associated pair of embeddings $\iota_x^\pm\colon \calO_{K_x, s_x}\arr \End(\bar F)$ extend $\bar \iota_d^\pm$. Let $K$ be the representative isomorphic to $\QQ_p[d]$. Then,
\begin{equation}
\begin{split}
\text{Fix}(d) & = \bigcup_{s\leq s(d)} QC(\bar \iota_d^\pm, K, s)\\& =  \bigcup_{s\leq s(d)} \{ x\in X: K_x \cong \QQ_p[d], s_x=s, \bar \iota_x^\pm = \bar \iota_d^\pm \}.
\end{split}
\end{equation}

We return to periodic points for $\beta(b)$, assuming $b$ is of infinite order in $\scrR^\times/\ZZ_p^\times$. It's an easy exercise that $\sup\{ s(b^r): r\geq 1\} = \infty$ . Let $K$ be the representative isomorphic to $\QQ_p[b]$. Then,
\begin{equation}
\begin{split}
 \text{Per}(b) &= \bigcup_{r\geq 1} \text{Fix}(b^r) \\&= \bigcup_{s\geq 0} QC(\bar \iota_b^\pm, K, s) \\& = \{ x\in X: K_x \cong K, \bar \iota_x^\pm  = \bar\iota_b^\pm\}
 \end{split}
\end{equation} 
(the abuse of notation $\bar \iota_b = \bar \iota_{b^r}$ is justified). Some information about the location of the periodic points can be found in \cite{qcan, GH}.

\subsubsection{} We can draw some immediate conclusions concerning the action of $\scrH_N$ on $X$. For one, for every element $b$ of $\scrH_N$, viewed as a branch of the Hecke operator $T_\ell^n$ for some $n$, the fixed and periodic points of $\beta(b)$ are given by \S\S \ref{subsec: 1} - \ref{subsec: 2}. In addition, in contrast to the ordinary case:

\begin{cor} The elements of $\scrA_2 = \scrR^\times$ do not have a common fixed point. In fact there is no finite subset of $X$ that is stable under $\scrH_N$.
\end{cor}
\begin{proof}The first statement follows easily, since we may pick integral elements $b_1$, $b_2$, generating distinct quadratic extensions inside $\scrB$. Then $\text{Fix}(b_1) \cap \text{Fix}(b_2) = \emptyset$. In fact, as in the proof of Proposition~\ref{prop: properties of scrH_n}, we can choose $b_i \in \scrH_N$ such that $b_i^{n_i} \not\in \ZZ_p^\times$, for all $n_i>0$.

As to the second statement, if $S$ is a finite set stable under $\scrH_N$, then, for suitable $n_i>0$, $S \subseteq \text{Fix}(b_i^{n_i}), i = 1,2$. But, we still have $\text{Fix}(b_1^{n_1}) \cap \text{Fix}(b_2^{n_2}) = \emptyset$. Thus, $S = \emptyset$.
\end{proof}

\subsection{The Gross-Hopkins period map} We specialize the results of Gross-Hopkins to the case $n=2$. Recall our notation from \S~\ref{subsec: notation for 5}: $q = p^2$, $\QQ_q$ is the quadratic unramified extension of $\QQ_p$. 

\vspace{0.2cm} 

\id In their paper \cite{GH} Gross and Hopkins construct a period map
\[ \Phi\colon  X \arr \PP^1\]
of rigid-analytic spaces over $\QQ_p$ with the following properties (loc. cit., in particular Lemma 19.3 and Corollary 23.15): 
\begin{enumerate}
\item $\Phi$ is a morphism of rigid spaces that is \'etale and surjective. 
\item $\Phi$ is equivariant for the action of $\scrA_2$, where $\scrA_2 = \{ \left( \begin{smallmatrix} a & pb^\sigma\\ b & a^\sigma \end{smallmatrix} \right)\colon   a\in  W(\FF_{q})^\times, b \in W(\FF_{q})\}\subset \GL_2(\QQ_{q})$, acting through M\"obius transformations on $\PP^1$. 
\item Furthermore, $\Phi$ restricts to a rigid-analytic isomorphism
\[ J\colon  = \{ x\in X\colon  \ord(x) \geq 1/2\} \overset{\sim}{\Arr} U\colon =\{ w\colon  \ord(w) \geq 1/2 \} \subset \PP^1_{(w_0\colon w_1)}\]
In particular,
the map $\Phi$ induces a bijection between the $\QQ_{p^2}$ points of $J$ and $U$.
Under this map the point $0$, corresponding to the formal group $F_0$, is mapped to the point $w=0=(1:0)$ (sic!), where $w = w_1/w_0$. Furthermore, both $J$ and $U$ are $\scrA_2$-invariant.  
\end{enumerate}

\subsubsection{Reduction to the ideal disc} Above we have discussed the work of Gross-Hopkins that relates to the deformation of a very particular formal group that we denoted $\bar F$. To connect it with our main subject, let $E$ be a supersingular elliptic curve defined over a finite extension $k$ of $\FF_{q}$, and $P$ a point of order  $N$ on $E$. Then $E$ has a unique model $E_1$ over $\FF_{q}$ whose Frobenius endomorphism $\varphi$ satisfies $\varphi^2 = p$ (Lemma~\ref{lem: models of elliptic curves over fpsquare}). Using \cite[Proposition 24.2.9]{Haz}, we find that the formal group $\widehat E_1$ is isomorphic to $\bar F$ over $\FF_{q}$, as the property $\varphi^2 = p$ holds for $\bar F$ as well. Fix such an isomorphism. Then the universal deformation spaces of $\uE=(E,P)$ and $\bar F$ become isomorphic after base change to $W(k)$. We remark that the functor $\scrF^\star$ in~(\ref{eqn: scrFstar}) (in the case $n=2$) is naturally equivalent to $\scrF$ in~(\ref{eqn: scrF}) -- see \cite[Proposition 3.3]{LT} and that allows us to use \cite{GH, LT}.

The associated rigid space $X_\uE = D(\uE)$ 
inherits  an action of $\scrA_2$ and an equivariant period map $\Phi\colon  X_\uE \arr \PP^1$ over $W(k)[p^{-1}]$, as well as a base point $0$ corresponding to the unique $A_2$-formal module -- none other than $F_0$ -- lifting $\widehat E$.   

\begin{lem}\label{lem: fields of definition of singular moduli} Let $K$ be a quadratic imaginary field in which $p$ is either inert or ramified. Let $E/\overline \QQ_p$ be an elliptic curve with CM by an order of $K$ of conductor prime to $p$. Then $\QQ_p(j(E))$ is a finite extension of $\QQ_p$ of ramification index at most $2$. In fact, if $p$ is inert in $K$, $\QQ_p(j(E))$ is at most a quadratic unramified extension of $\QQ_p$, while if $p$ is ramified in $K$, $\QQ_p(j(E))$ has degree dividing $4$ over $\QQ_p$.
\end{lem}
\begin{proof} Assume that $E$ has $CM$ by $R_\gerf$, the order of conductor $\gerf \in \ZZ_{\geq 1}$ in $\calO_K$. Let $K(\gerf)$ denote the ring class field of conductor $\gerf$ over $K$. Let $V$ denote the set of all finite places of $K$, and for $v\in V$, denote by $O_v$ (resp., $K_v$) the completion of $\calO_K$ (resp., $K$) at $v$. Let $\pi_v\colon \calO_v^\times \arr (\calO_v/\gerf\calO_v)^\times$ be the natural projection.  Under the isomorphism of class field theory, $K(\gerf)$ corresponds to the subgroup $K^\times W \subset \AA_K^\times$, where $W$ is given by $W=\Pi_{v\in V} W_v$, and $W_v$ is defined as
\[
W_v=\begin{cases} \pi_v^{-1}((\ZZ/\gerf\ZZ)^\times)& v|\gerf, \\ \calO_{K,v}^\times   & v\!\!\!\not |\gerf. \\\end{cases}
\]
 Let $\gerp$ denote the unique prime of $K$ above $p$, and choose a prime $\tilde{\gerp}$ of $K(\gerf)$ above $\gerp$. By class field theory, the local field $K(\gerf)_{\tilde{\gerp}}$ corresponds to the subgroup $K^\times W \cap K_\gerp \subset K_\gerp^\times$. Let us denote by $\Pic(R_\gerf)$ the ideal class group of $R_\gerf$. Let $a$ denote the order of $[\gerp]$ in $\Pic(R_\gerf)$. Then 
 \[
 K^\times W \cap K_\gerp =\calO_\gerp^\times\varpi_\gerp^{a\ZZ},
 \]
where $\varpi_\gerp$ is a choice of a uniformizer in $\calO_\gerp$. In particular, we find that $K(\gerf)_{\tilde{\gerp}}$ is an unramified extension of degree $a$ over $K_\gerp$. By the theory of complex multiplication we have $K(j(E))=K(\gerf)$. Hence, it is enough to show that  $K(\gerf)_{\tilde{\gerp}}$ satisfies the conditions stated in the lemma. If $p$ is inert in $K$, then $K_\gerp=\QQ_{p^2}$ and $a=1$. If $p$ is ramified in $K$, then $K_\gerp$ is a quadratic  ramified extension of $\QQ_p$ and $a\leq 2$. This proves the lemma.
\end{proof}
\begin{cor}\label{cor: singular moduli in J}All CM points of conductor prime to $p$ in the residue disc $D(\uE)$ lie in $J(\QQ_q)$. 
\end{cor}

\subsection{Dynamics implications-II: Ideas from ergodic theory}   Corollary~\ref{cor: singular moduli in J} asserted that all  CM points of discriminant prime to $p$ in $D(\uE)$ lie in the set $J$ defined above. Clearly, if we wish to understand the action of $\scrA_2$ on $J$ -- and in particular on such CM points -- we may as well study the action of $\scrA_2$ on $U$. More precisely, we are interested in the action of $\scrH_N$ on~$J$ and we may reduce it to the very explicitly described action of $\scrG_N$ on $U$ -- see Corollary~\ref{cor: description of G_N}.

This is a fortunate setting, achieved at the price of restricting our attention to ``mildly ramified" CM points that lie in a particular small disc within $D(\uE)$. The general case is more complex and indeed the map $\Phi$ is given by a complex formula (\cite{GH}, equations (25.6), (25.7)) that makes it hard to translate results for $\PP^1$ to results for $D(\uE)$.

\subsubsection{Minimal sets} Let $\KK$ be a field extension of $\QQ_q$. We denote $J(\KK), U(\KK)$, etc. points of $J, U,$ etc. with coordinates in $\KK$. Recall that $\HH_N$ is the closure of $\scrH_N$ and of $\scrH_N^+$ in $\scrR^\times = \scrA_2$. The groups $\HH_N$, $sat(\HH_N)$, $\GG_N$ are the same in their action on $\PP^1(\KK)$, and in fact act by isometries on $U(\KK)$. This follows from the identity $ \gamma x - \gamma y = \frac{(x-y)(aa^\sigma - pbb^\sigma)}{(bx+ a^\sigma)(by + a^\sigma)}$, for $\gamma = \left( \begin{smallmatrix} a & pb^\sigma \\ b & a^\sigma
\end{smallmatrix}\right)$.

An \textit{$\HH_N$-minimal set} of $\PP^1(\KK)$  is an $\HH_N$-invariant non-empty closed set that is minimal relative to these properties. 
\begin{prop}\label{prop: Gamma mu orbits} The $\HH_N$-minimal sets $S$ are the sets of the form $\HH_N\cdot f$, for some $f\in \PP^1(\KK)$. Furthermore, $\HH_N \cdot f = \overline{\scrH_N^+\cdot f}$.
\end{prop}
\begin{proof} First note that for any $f$ we have that $\HH_N \cdot f$ is closed - in fact, compact - because $\HH_N$ is compact and the action of $\GL_2(\KK)$ on $\PP^1(\KK)$ is continuous.
If $S$ is minimal, pick any $f\in S$, then $\HH_N \cdot f \subseteq S$ is closed and $\HH_N$-invariant, hence equal to $S$. Conversely, consider a set $\HH_N\cdot f$ for some $f\in \PP^1(\KK)$. As an $\HH_N$-minimal subset contained in $\HH_N\cdot f$ contains an element of the form $\gamma f$ for some $\gamma \in \HH_N$ and so it also contains $\HH_N \cdot \gamma f = \HH_N\cdot f$.

To prove the rest, one proves that $\overline{\scrH_N^+\cdot f} = \overline{\scrH_N^+}\cdot f = \HH_N \cdot f$: On the one hand, $\overline{\scrH_N^+}\cdot f$ is closed and contains $\scrH_N^+\cdot f$, hence also its closure. On the other hand, let $\gamma_i \in \scrH_N^+, \gamma_i \arr \gamma$, then also $\gamma_i \cdot f \arr \gamma \cdot f$. This shows that $ \overline{\scrH_N^+}\cdot f  \subseteq\overline{\scrH_N^+\cdot f} $.
\end{proof}

\subsubsection{Measures} Our proof in Proposition~\ref{prop: properties of scrH_n} that $\scrG_N$ is finitely generated in fact supplied us with a set of $t = t(N)$ elements, \[ \Gamma = \{ \gamma_1, \dots, \gamma_{t} \},\] in $\scrH_N^+$ that  generate the image of $\scrH_N^+$ in $\PGL_2(\QQ_q)$ as a monoid. We include the identity among them. Note that there is no canonical choice of such elements, though we chose the generators to be the set of all elements of $\scrH_N^+$
obtained from closed labeled walks of explicitly bounded length.  Let $\mu = \mu(N, \Gamma)$ be the atomic measure of $\GL_2(\QQ_{q})$ supported on $\Gamma$. Namely, 
\[ \mu  = \frac{1}{t}\sum_{i=1}^t \delta_{\gamma_i}.\]
The following lemma is  clear. 
\begin{lem} The group $\HH_N$ is the minimal closed monoid of $\GL_2(\QQ_q)$ (in the $p$-adic topology) containing the support of $\mu$. To underline that, denote it in the sequel $\HH_{N, \mu}$.
\end{lem}

We use the work of Benoist and Quint \cite{BQ1}; see also the references \cite{BQ2, EW} for general background material. In particular,  \cite{BQ1} implies the following:

\begin{thm} The map $\nu \mapsto \text{ Supp}(\nu)$ is a bijection 
\[ \{ \mu-\text{ergodic probability measures on } U(\KK)\} \Leftrightarrow \{ \HH_{N, \mu}-\text{minimal subsets of } U(\KK)\}.\]
\end{thm}
\begin{cor}\label{cor: 5104} There is a unique $\HH_{N,\mu}$-ergodic measure $\nu(\mu)$ on $U(\QQ_q)$, hence a unique $\HH_{N, \mu}$-stationary probability measure. Its support is $U(\QQ_q)$.
\end{cor} 
\begin{proof} The $\mu$-ergodic measures on $U(\QQ_{q})$ are in bijection with $\HH_{N, \mu}$-minimal sets. We claim that $\scrR^\times_\square$ acts transitively on $U(\QQ_q)$: given $x\in U(\QQ_{q})$, for any $a\in W(\FF_q)^\times$ the matrix $\left(\begin{smallmatrix} a & p(xa)^\sigma \\ xa & a^\sigma\end{smallmatrix}\right)$ takes $(1, 0)$ to $x$ and has determinant $aa^\sigma (1 - pxx^\sigma)$ which for an appropriate choice of $a$ lies in $(\ZZ_p^\times)^2\langle \ell \rangle$. That is, it belongs to $\scrR^\times_\square$.

As $\scrR^\times_\square$ acts transitively on $U(\QQ_q)$ we can write $U(\QQ_q)=\scrR^\times_\square /B$, where $B$ is the stabilizer of $(1, 0)$ in $\scrR^\times_\square$ and it contains $\ZZ_p^\times$. Thus, by Proposition~\ref{prop: Gamma mu orbits}, the minimal orbits are in bijection with $\HH_{N, \mu} \backslash \scrR^\times_\square /B = sat(\HH_{N, \mu}) \backslash \scrR^\times_\square /B$, which is a singleton. The ergodic measures are the extremal points of the convex closed set of probability measures on the compact metric space $U(\QQ_q)$ and the Krein-Milman theorem (see e.g., \cite{Sim} Chapter 8) then provides the uniqueness of the $\HH_{N, \mu}$-stationary probability measure.
\end{proof}

By Proposition~\ref{prop: Zariski closure} the Zariski closure $H_N$ of $\scrH_N^+$ is $\GL_2(\QQ_{q})$, and is in particular reductive. We may therefore apply the results of Benoist and Quint \cite{BQ1} and conclude the following. 
\begin{thm} For every $x\in U(\KK)$ the measure
\begin{equation}
\label{eqn: limit measure weak}
 \nu_x = \lim_{n\arr \infty} \frac{1}{n+1} \sum_{k=0}^n \mu^{\ast k}\ast \delta_x
 \end{equation}
exists and is a $\mu$-stationary measure, depending continuously on $x$. Thus, if $x\in U(\QQ_q)$, $\nu_x = \nu(\mu)$ and is independent of $x$.
\end{thm}
Note that $\sum_{k=0}^n \mu^{\ast k}\ast \delta_x$ is nothing but the images of $x$ (counted with multiplicities) under the composition of up to $n$ automorphisms of the form $\gamma_1^{-1}, \dots, \gamma_t^{-1}$. This theorem is the measure-theoretic manifestation that the orbit of $x$ is ``well-distributed" in the closure of its orbit. Note that while the measure depends on $\Gamma$, the support of the measure, i.e. $\HH_{N, \mu}\cdot x$, does not. Thus, whatever ``branches" of the Hecke operator $T_\ell$ we choose, the orbit is well-distributed in its closure that is independent of $N$.

While the theorem above can viewed as a non-commutative analogue of the weak law of large numbers, the next result is the analogue of the strong law of large numbers. 
\begin{thm} For every $x\in U(\KK)$, for almost all sequences $\gamma = (\gamma_{i_n})_{n=1}^\infty$, $\gamma_{i_n} \in \{ \gamma_1, \dots, \gamma_t\}$, the measure
\begin{equation}
\label{eqn: limit measure strong} \nu_{x, \gamma} = \lim_{n\arr \infty}  \frac{1}{n} \sum_{k=1}^n \delta_{\gamma_{i_n}^{-1} \cdots \gamma_{i_1}^{-1}x}
\end{equation}
exists and is a $\mu$-stationary $\mu$-ergodic probability measure. Furthermore, 
\[ \nu_x = \int \nu_{x, \gamma} d\gamma.\]
In particular, if $x\in U(\QQ_q)$ 
\[ \nu_{x, \gamma} = \nu_x = \nu(\mu).\]
\end{thm}

Finally, we recall the implications for CM points. 
\begin{thm}\label{thm:5107} Let $x = \uE = (E, P)$ be a CM point of supersingular reduction on $X_1(N)$, $N\geq 3$, such that $\End(E)$ has discriminant prime to $p$. Let $\scrH_N^+ = \scrH_N^+(\uE)$ denote the set of branches of the iterations of the Hecke operator $T_\ell$ returning to the residue disc $D = D(\uE)$. Let $\HH_{N, \mu}$ be its $p$-adic closure equipped with the probability measure $\mu$ relative to the generators $\{ \gamma_1, \dots, \gamma_t\}$ as defined above.

There is a unique $\HH_{N,\mu}$-invariant (ergodic) probability measure $\nu(\mu)$ on $J(\QQ_q)$, and we have
\begin{equation}\label{eqn: final measure}
\begin{split}
 \nu(\mu) &= \lim_{n\arr \infty} \frac{1}{n+1} \sum_{k=0}^n \mu^{\ast k}\ast \delta_x\\
 & = \lim_{n\arr \infty}  \frac{1}{n} \sum_{k=1}^n \delta_{\gamma_{i_n}^{-1} \cdots \gamma_{i_1}^{-1}x}, 
\end{split}
\end{equation}
for almost all sequences $\gamma = (\gamma_{i_n})_{n=1}^\infty$, $\gamma_{i_n} \in \{ \gamma_1, \dots, \gamma_t\}$.
\end{thm}

\

\id {\bf Acknowledgements.} This work was done during visits to King's College London and McGill university and we thank both institutions for their hospitality. The support of NSERC through its Discovery Grant program and the London Mathematical society through its Research Grant is gratefully acknowledged. We thank an anonymous referee of the first version of this paper for excellent suggestions that allowed us to strengthen our results in the supersingular case and improve the overall presentation. We thank the referee of the second version for many useful comments and their thorough reading of the manuscript. As a result of their work, this paper is significantly better than the original version.


\end{document}